%% file: effective-contraction-v4-arxiv.tex
\documentclass[11pt,a4paper]{article}
\usepackage{amsthm, amssymb, amsmath, a4wide}
\usepackage[all]{xy} 
\usepackage{paralist} 
\usepackage{graphicx} 
\usepackage[sans]{dsfont} 
\usepackage{varioref} 
\usepackage{subfigure}
\usepackage{tikz}
\usepackage{pgf}
\usetikzlibrary{decorations.pathreplacing,calc}
\usepackage{xifthen}
\usepackage{hyperref}

\title{An effective criterion for algebraicity of rational normal surfaces}
\author{Pinaki Mondal\\ Weizmann Institute of Science\\ pinaki@math.toronto.edu}

\input{commands-2012}

\newcommand{\xdelta}{\bar X^\delta}
\newcommand{\hot}{\text{h.o.t.}}
\newcommand{\lot}{\text{l.o.t.}}

\newcommand{\psx}[2]{{#1} \{ \{ #2 \} \} }
\newcommand{\psxcu}{\psx{\cc}{u}}

\newcommand{\dpsx}[2]{{#1} \langle \langle #2 \rangle \rangle }
\newcommand{\dpsxc}{\dpsx{\cc}{x}}

\def\noqed{\renewcommand{\qedsymbol}{}}

\begin{document}
\maketitle

\begin{abstract} 
We give a novel and effective criterion for algebraicity of rational normal analytic surfaces constructed from resolving the singularity of an irreducible curve-germ on $\pp^2$ and contracting the strict transform of a given line and all but the `last' of the exceptional divisors. As a by-product we construct a new class of analytic non-algebraic rational normal surfaces which are `very close' to being algebraic. These results are {\em local} reformulations of some results in \cite{sub2-2} which sets up a correspondence between normal algebraic compactifications of $\cc^2$ with one irreducible curve at infinity and algebraic curves in $\cc^2$ with one {\em place} at infinity. This article is meant partly to be an exposition to \cite{sub2-2} and we give a proof of the correspondence theorem of \cite{sub2-2} in the `first non-trivial case'.
\end{abstract}

\section{Introduction} \label{sec-intro}

 
In \cite{sub2-2} we give an explicit criterion to determine when a normal analytic compactification of $\cc^2$ with an irreducible curve at infinity of $X$ is algebraic. The geometric counterpart of this criterion is a correspondence between the following categories of objects:
\begin{align} \label{fundamental-correspondence}
\begin{tabular}{p{7cm}cp{4.5cm}}
normal algebraic compactifications of $\cc^2$ with one (irreducible) curve at infinity
& $\longleftrightarrow$ 
& algebraic curves in $\cc^2$ with one {\em place} at infinity 
\end{tabular}
\end{align}
(recall that `one place' means only one branch which is analytically irreducible). In this article we reformulate the results in the {\em local setting} and describe some of the (hopefully interesting) consequences. We also give a proof of the results under a (greatly) simplifying assumption (which however applies to most of the examples we consider in this article). The paper consists of two parts which can be read more or less independently: the (rest of the) Introduction and Section \ref{sec-local} deals with the local setting, whereas in Sections \ref{sec-global} and \ref{sec-proof} we give a complete statement of the main correspondence result in the global setting and give a proof under the simplifying assumption mentioned above.

\subsection{Introduction to the problem} \label{subsec-algebraic-introduction}

Fix a line $L \subseteq \pp^2$ and let $\pi: Y \to \pp^2$ be a birational morphism of non-singular algebraic surfaces.  Fix an irreducible component $E^*$ of the exceptional divisor $E$ of $\pi$, and let $\tilde E$ be the union of the strict transform $\tilde L$ of $L$ with all components of $E$ except for $E^*$. 

\begin{bold-question} \label{local-existential-question}
When is $\tilde E$ {\em contractible}, i.e.\ when does there exist a proper surjective morphism $\tilde \pi: Y \to \tilde Y$ of normal analytic surfaces such that $\tilde \pi(\tilde E)$ is a point in $\tilde Y$ and $\tilde \pi$ restricts to an isomorphism on $Y \setminus \tilde E$? 
\end{bold-question}

{
\addtocounter{thm}{-1}
\let\oldthm\thethm
\renewcommand{\thethm}{\oldthm$'$}
\begin{bold-question}\label{local-algebraic-question}
When is $\tilde E$ {\em algebraically contractible}, i.e.\ when does there exist $\tilde Y$ as in the preceding question such that $\tilde Y$ is also algebraic?
\end{bold-question}
}

It follows from a criterion of Grauert \cite{grauert} that the answer to Question \ref{local-existential-question} is affirmative iff the matrix of intersection numbers of the irreducible components of $\tilde E$ is negative definite, or equivalently, as we showed in \cite{sub2-1}, iff the valuation corresponding to $E^*$ is {\em positively skewed} in the sense of \cite{favsson-eigen} as a valuation {\em centered at infinity} with respect to $\pp^2 \setminus L$ (see Proposition \ref{local-existential-answer} for an {\em explicit} version). In particular, the answer to Question \ref{local-existential-question} depends only on the configuration of the curves in $\tilde E$. The answer to Question \ref{local-algebraic-question} however is more delicate, as the following example shows. 

\begin{example}\label{non-example}
Consider the set up of Question \ref{local-algebraic-question}. Let $O \in L \subseteq \pp^2$ and $(u,v)$ be a system of {\em affine coordinates} at $O$ (`affine' means that both $u=0$ and $v=0$ are lines on $\pp^2$) such that $L=\{u=0\}$. Let $C_1$ and $C_2$ be curve-germs at $O$ defined respectively by $f_1 := v^5 - u^3$ and $f_2 := (v-u^2)^5 - u^3$. Note that $C_j$'s are isomorphic as curve-germs via the map $(u,v) \mapsto (u,v-u^2)$. For each $i$, Let $Y_i$ be the surface constructed by resolving the singularity of $C_i$ at $O$ and then blowing up $8$ more times the point of intersection of the strict transform of $C_i$ with the exceptional divisor. Let $E^*_i$ be the {\em last} exceptional divisor, and $\tilde E_i$ be the union of the strict transform $\tilde L_i$ (on $Y_i$) of $L$ and (the strict transforms of) all but the last of the exceptional divisors. It is straight-forward to check that both $\tilde E_i$ have the {\em same} dual graph (i.e.\ the graph whose vertices are the irreducible components of $\tilde E_i$ and there is an edge between two vertices iff corresponding curves intersect) and are analytically contractible. Figure \ref{fig:non-example} depicts the dual graph of $\tilde E_i$; note that we labelled the vertices according to the order of appearance of corresponding curves in the sequence of blow-ups. Below we list some other common properties of $\tilde E_1$ and $\tilde E_2$. 

\begin{figure}[htp]
\begin{center}

\begin{tikzpicture}[scale=1, font = \small] 

	\pgfmathsetmacro\factor{1.25}
 	\pgfmathsetmacro\dashedge{4.5*\factor}	
 	\pgfmathsetmacro\edge{.75*\factor}
 	\pgfmathsetmacro\vedge{.55*\factor}

 	\draw[thick] (-2*\edge,0) -- (\edge,0);
 	\draw[thick] (0,0) -- (0,-2*\vedge);
 	\draw[thick, dashed] (\edge,0) -- (\edge + \dashedge,0);
 	\draw[thick] (\edge + \dashedge,0) -- (2*\edge + \dashedge,0);
 	
 	\fill[black] ( - 2*\edge, 0) circle (3pt);
 	\fill[black] (-\edge, 0) circle (3pt);
 	\fill[black] (0, 0) circle (3pt);
 	\fill[black] (0, -\vedge) circle (3pt);
 	\fill[black] (0, - 2*\vedge) circle (3pt);
 	\fill[black] (\edge, 0) circle (3pt);
 	\fill[black] (\edge + \dashedge, 0) circle (3pt);
 	\fill[black] (2*\edge + \dashedge, 0) circle (3pt);
 	
 	\draw (-2*\edge,0)  node (e0up) [above] {$-1$};
 	\draw (-2*\edge,0 )  node (e0down) [below] {$\tilde L$};
 	\draw (-\edge,0 )  node (e2up) [above] {$-3$};
 	\draw (-\edge,0 )  node [below] {$E_2$};
 	\draw (0,0 )  node (e4up) [above] {$-2$};	
 	\draw (0,0 )  node [below left] {$E_4$};
 	\draw (0,-\vedge )  node (down1) [right] {$-2$};
 	\draw (0,-\vedge )  node [left] {$E_3$};
 	\draw (0, -2*\vedge)  node (down2) [right] {$-3$};
 	\draw (0,-2*\vedge )  node [left] {$E_1$};
 	\draw (\edge,0)  node (e+1-up) [above] {$-2$};
 	\draw (\edge,0)  node [below] {$E_5$};
 	\draw (\edge + \dashedge,0)  node (e-last-1-up) [above] {$-2$};
 	\draw (\edge + \dashedge,0)  node [below] {$E_{10}$};
 	\draw (2*\edge + \dashedge,0)  node (e-last-up) [above] {$-2$};
 	\draw (2*\edge + \dashedge,0)  node [below] {$E_{11}$};
 	
 	\pgfmathsetmacro\factorr{.2}
 	\draw [thick, decoration={brace, mirror, raise=15pt},decorate] (\edge + \edge*\factorr,0) -- (\edge + \dashedge - \edge*\factorr,0);
 	\draw (\edge + 0.5*\dashedge,-\edge) node [text width= 5cm, align = center] (extranodes) {string of vertices of weight $-2$};
 			 	
\end{tikzpicture}
\caption{Dual graph of $\tilde E_i$}\label{fig:non-example}
\end{center}
\end{figure}
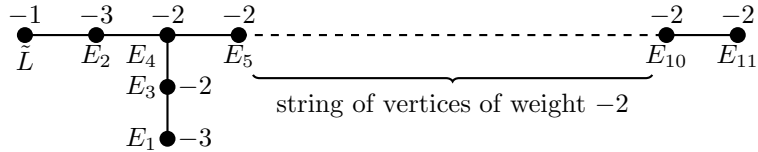

\begin{compactenum}
\item Removing (from the dual graph) the vertex corresponding to $E_{11}$ turns it into the resolution graph of a {\em rational} singularity. 
\item \label{sandwiched-property} Removing the vertex corresponding to $\tilde L$ turns it into the resolution graph of a {\em sandwiched} singularity (which is a special class of rational singularities - see \cite{spinash}). 
\item The normal analytic surface $\tilde Y_i$ constructed from blowing down $\tilde E_i$ has a trivial canonical sheaf and a unique singular point which is {\em almost rational} in the sense of \cite{nemethi}.
\end{compactenum}
However, it turns out that $\tilde Y_1$ is algebraic, but $\tilde Y_2$ is {\em not} (see Example \ref{non-example-again}). In the algebraic case, it can also be shown that the image of $E^*_1$ on $\tilde Y_1$ is non-singular; we don't know what happens in the non-algebraic case.
\end{example}

Note that Property \ref{sandwiched-property} of the resolution graph of Figure \ref{fig:non-example} in fact holds true in general in the set up of Question \ref{local-algebraic-question} (so that the singularities on $\tilde Y$ resulting from contraction of $\tilde E$ are {\em almost sandwiched}). Indeed, removing the vertex corresponding to $\tilde L$ and then adding a vertex corresponding to $E^*$ produces the dual graph of $E$ (using the notation of the set up of Question \ref{local-algebraic-question}), which is simply the exceptional divisor of $\pi$. Then Property \ref{sandwiched-property} follows from the definition of sandwiched singularities \cite[Definition 1.9]{spinash}, namely a singularity is sandwiched iff the dual graph of its resolution is a part of the dual graph of the exceptional divisor of a morphism between non-singular surfaces.  \\

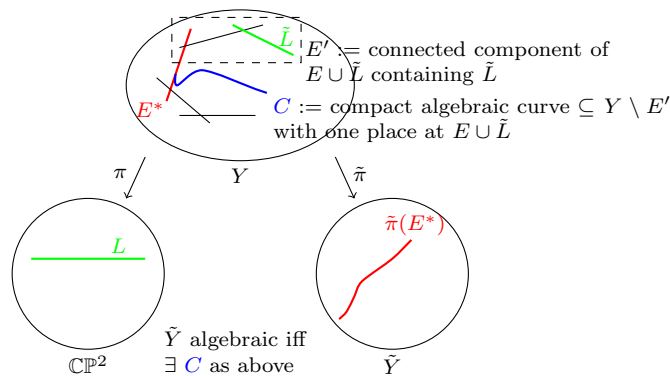
\begin{figure}[htp]
		\begin{center}
			\begin{tikzpicture}[font = \scriptsize]
			\pgfmathsetmacro\scalefactor{1}
			\pgfmathsetmacro\centerdist{2*\scalefactor}

			\pgfmathsetmacro\distop{2.5*\scalefactor}
			\pgfmathsetmacro\arrowbelowdist{0.75*\scalefactor}

			\pgfmathsetmacro\arrowdiagtopydist{1.55*\scalefactor}
			\pgfmathsetmacro\arrowdiagtopxdist{1.25*\scalefactor}

			\pgfmathsetmacro\arrowdiagbottomydist{1*\scalefactor}
			\pgfmathsetmacro\arrowdiagbottomxdist{1.5*\scalefactor}

			\begin{scope}[shift={(\centerdist,0)}, scale=\scalefactor]

				\draw[] (0,0) circle (1cm);
				
				\draw [red, thick] plot [smooth] coordinates {(.25,0.45) (0,0.2) (-0.4,-0.1) (-0.5, -0.3) (-0.6, - 0.5) (-0.7,-0.6)};

				\draw [red] (0.3,0.65) node {$\tilde \pi(E^*)$};
				\draw (0,-1.2) node {$\tilde Y$};
			\end{scope}

			\begin{scope}[shift={(-\centerdist,0)}, scale=\scalefactor]

				\draw[] (0,0) circle (1cm);
				\draw [green,thick] (-0.75, 0.2) -- (0.75, 0.2);


				\draw [green] (0.4,0.375) node {$L$};
				\draw (0, -1.2) node {$\cc\pp^2$};
			\end{scope}

			\begin{scope}[shift={(0, \distop)}, scale=\scalefactor]

				\draw[] (0,0) ellipse (1.5cm and 1cm);
				\begin{scope}[shift= {(-0.4, 0.4)}] 
					\draw [red, thick] plot [smooth] coordinates {(-0.25,0.35) (-0.4,-0.1) (-0.45, -0.25) (-0.57,-0.61)};
					\draw (-0.79,-0.72) node[red] {$E^*$};
				\end{scope}

				\draw [] (-0.8, 0.5) -- (0.3,0.8);
				\draw [green, thick] (-0.1, 0.8) -- (0.7, 0.4);
				\draw (0.6,0.65) node [green]{$\tilde L$};
				\draw [dashed] (-0.9,0.3) rectangle (0.8, 0.9);
				\draw (0.72,0.3) node [right, text width = 4cm] {$E':=$ connected component of $E \cup \tilde L$ containing $\tilde L$}; 
				\draw [blue,thick] plot [smooth] coordinates {(-0.85, 0.15) (-0.83,-0) (-0.5,0.2) (0.2, -0.05) (0.35,-0.1)};
				\draw (0.3,-0.45) node[right, text width = 5.5cm]{\textcolor{blue}{$C$} $:=$ compact algebraic curve $\subseteq Y \setminus E'$ with one place at $E \cup \tilde L$};

				\draw [] (-1.1, 0.1) -- (-0.4,-0.5);
				\draw [] (-0.8, -0.4) -- (0.2, -0.4);
				\draw (0, -1.2) node {$Y$}; 
			\end{scope}

			\draw[->] (-\arrowdiagtopxdist,\arrowdiagtopydist) -- (-\arrowdiagbottomxdist, \arrowdiagbottomydist) node [left, pos=0.4] {$\pi$};

			\draw[->] (\arrowdiagtopxdist,\arrowdiagtopydist) -- (\arrowdiagbottomxdist, \arrowdiagbottomydist) node [right, pos=0.4] {$\tilde \pi$};
			\draw (0, -1) node[text width = 2cm]{$\tilde Y$ algebraic iff $\exists$ \textcolor{blue}{$C$} as above};
			
			\end{tikzpicture}
			\caption{Geometric answer to Question \ref{local-algebraic-question}} \label{fig:geom-answer}
		\end{center}
\end{figure}

We give two versions of the answer to Question \ref{local-algebraic-question}: a {\em geometric}, but {\em non-effective} version (Theorem \ref{algebraic-answer-1}) as depicted in Figure \ref{fig:geom-answer} and an {\em effective} version; to avoid being redundant, we state (and prove) the effective version only for the simplest case (Theorem \ref{local-algebraic-answer}) and give the complete statement only for the {\em global} version (Theorem \ref{algebraic-global}). The effective answer is especially useful to construct new\footnote{The examples in the existing literature (that we know of) of constructions (e.g.\ in \cite{grauert}) of non-algebraic normal Moishezon surfaces involve contraction of {\em non-rational} curves from smooth surfaces. On the other hand, the non-algebraic surfaces emanating from negative answers to Question \ref{local-algebraic-question} come from contraction of {\em rational trees}.} classes of non-algebraic analytic (normal) rational surfaces - see also Remark \ref{1-puiseux-remark} and Remark-Example \ref{non-2-remexample}. Since having only rational singularities implies algebraicity \cite{artin-rational}, Example \ref{non-example} (and the paragraph following Example \ref{non-example}) shows that in a sense these surfaces are `very close' to being algebraic.\\

It is well known (and also illustrated by Example \ref{non-example}) that in general algebraicity can not be determined only from the dual graph of the exceptional divisor of the resolution of singularities. However, in the set up of Question \ref{local-algebraic-question} we can completely classify (in terms of two {\em semigroup conditions}) dual graphs of $\tilde E$ which correspond to {\em only algebraic} contractions, those which correspond to {\em only non-algebraic} contractions, and those which correspond to {\em both} types of contractions (Theorem \ref{semigroup-prop}). \\

The problem of determining algebraicity of a (compact) analytic surface (or more generally, variety) $Y$ has been extensively studied. \cite[Satz 2]{grauert} gives a criterion in terms of the existence of a positive holomorphic line bundle on $Y$. On the other hand, a necessary requirement for $Y$ to be algebraic is that the transcendence degree over $\cc$ of the meromorphic function field of $Y$ should equal $\dim(Y)$, in which case $Y$ is said to be a {\em Moishezon space}. It was shown in \cite{artractability} that a normal Moishezon surface with at most {\em rational} singularities is projective. There have been a number of other works which give criteria for analytic surfaces to be algebraic, see e.g.\ \cite{morrow-rossi}, \cite{brenton-algebraicity}, \cite{franco-lascu}, \cite{schroe-traction}, \cite{badescu-contractibility}, \cite{palka-Q1}. However, all the criteria (for algebraicity) that appear in the literature are given in terms of cohomological or analytic invariants which are not suitable\footnote{We note however that there are numerical criteria (e.g.\ in \cite{artin-rational}) applicable in our setting to determine if the singularities of the surface $\tilde Y$ of Question \ref{local-algebraic-question} are rational - but in general (e.g.\ in Example \ref{non-example}) $\tilde Y$ will have non-rational singularities, so that these tests do not apply.} for examining the set-up of Question \ref{local-algebraic-question}. Our `geometric criterion' (see also Remark \ref{essence-1}) is stated in terms of the existence of a certain kind of divisors and has in a sense the same spirit as the criteria of \cite[Theorem 3.4]{schroe-traction} and \cite[Corollary 2.6]{palka-Q1}. Our `effective criterion' states that $\tilde Y$ of Question \ref{local-algebraic-question} is algebraic iff a certain element of $\cc[x,x^{-1},y]$ we compute from the input data of Question \ref{local-algebraic-question} is in fact an element of $\cc[x,y]$ (Theorem \ref{algebraic-global}). To our knowledge this type of criterion for algebraicity does not exist in the literature - it would certainly be interesting to relate it to classical invariants.\\

Finally, we point out that if we identify (in the set up of Question \ref{local-algebraic-question}) $\pp^2\setminus L$ with $\cc^2$, then the {\em geometric} criterion (Theorem \ref{algebraic-answer-1}) for algebraicity of $\tilde Y$ of Question \ref{local-algebraic-question} is precisely the existence of a certain algebraic curve in $\cc^2$ with one place at infinity on $\cc^2$ (it is more explicit in the global version - Theorem \ref{geometric-global}). Moreover, Theorem \ref{geometric-global} has an (almost immediate) translation in the terminology of {\em valuative tree} \cite{favsson-tree} which we now describe. In the set up of Question \ref{local-algebraic-question}, let $X := \pp^2 \setminus L \cong \cc^2$ and $\nu$ be the {\em divisorial valuation} (see Definition \ref{divisorial-defn}) on $\cc(X)$ corresponding to $E^*$. Choose polynomial coordinates $(x,y)$ on $X$. Then the {\em valuative tree at infinity} $\scrV_0$ on $\cc[x,y]$ is the space of all valuations $\mu$ on $\cc[x,y]$ such that $\min\{\mu(x),\mu(y)\} = -1$. It turns out that $\scrV_0$ has the structure of a tree with root at $-\deg_{(x,y)}$ \cite[Section 7.1]{favsson-eigen}, where $\deg_{(x,y)}$ is the usual degree in $(x,y)$-coordinates. Let $\tilde \nu := \nu/\max\{-\nu(x),-\nu(y)\}$ be the `normalized' image of $\nu$ in $\scrV_0$. 


\begin{thm}[A corollary of Theorem \ref{geometric-global}] \label{valuative-global}
Assume $\tilde Y$ of Question \ref{local-existential-question} exists. Then it is algebraic iff there is a tangent vector $\tau$ of $\tilde \nu$ on $\scrV_0$ such that 
\begin{compactenum}
\item $\tau$ is {\em not} represented by $-\deg$, and
\item $\tau$ is represented by a curve valuation corresponding to an algebraic curve with one place at infinity.
\end{compactenum}
\end{thm}

This correspondence between algebraicity of $\tilde Y$ and existence of plane curves with one place at infinity is also evident in the comparison of the {\em semigroup conditions}. More precisely, it is possible to encode the input data for Question \ref{local-algebraic-question} in terms of a curve-germ $C$ at $O$ and a positive integer $r$ (see Subsection \ref{subsec-effective}). Then we show that for a fixed $r$ and a fixed singularity type (of plane curve-germs), there is a curve-germ $C$ with the given singularity type such that the corresponding $\tilde Y$ is algebraic, iff the sequence of {\em virtual poles} (Definition \ref{virtual-defn}) satisfies a `semigroup condition'. On the other hand, it follows from a fundamental result (developed in \cite{abhya-moh-tschirnhausen}, \cite{abhyankar-expansion}, \cite{abhyankar-semigroup}, \cite{sathaye-stenerson}) of the theory of plane curves with one place at infinity that the {\em same} semigroup condition implies the existence of a plane algebraic curve $\tilde C$ with one place at infinity with `almost' the given singularity type at infinity. Moreover, if the curve $\tilde C$ exists, then the `virtual poles' are (up to a constant factor) precisely the generators of the {\em semigroup of poles} at the point at infinity of $\tilde C$ - i.e.\ in this case {\em virtual poles are real!} We refer to Subsection \ref{subsec-semigroup} for details.
 
\subsection{Acknowledgements}
I would like to thank Professor Peter Russell for the generosity with his time and very helpful remarks and Karol Palka for enlightening discussions. This exposition is in a large part motivated by their suggestions. 

\section{Algebraicity in the {\em local} setting} \label{sec-local}

\subsection{Geometric criterion for algebraicity} \label{subsec-geometric}
We use here the notations and set-up of Question \ref{local-algebraic-question} and give the {\em geometric} answer. 

\begin{thm}[Geometric criterion for algebraic contractibility] \label{algebraic-answer-1}
Assume $\tilde E$ is contractible, i.e.\ there exists $\tilde Y$ as in Question \ref{local-existential-question}. Then $\tilde Y$ is algebraic iff there is a compact algebraic curve $C \subseteq Y \setminus E'$ such that $C$ has {\em only one place at} $E \cup \tilde L$, where $E'$ is the connected component of $E \cup \tilde L$ that contains $E^*$ (see Figure \ref{fig:geom-answer}).
\end{thm}

\begin{rem} \label{rem-fundamental-correspondence}
The phrase `$C$ has only one place at $E \cup \tilde L$' (which is essentially the `essence' of Theorem \ref{algebraic-answer-1} - see Remark \ref{essence-1}) means that $C$ intersects $E \cup \tilde L$ at only one point $P$ and $C$ is analytically irreducible at $P$. Identifying $Y \setminus \left(E \cup \tilde L\right)$ with $\cc^2$, this is equivalent to saying that $C \cap \cc^2$ has {\em one place at infinity}. This observation sets up the correspondence \eqref{fundamental-correspondence} and provides the equivalence between Theorem \ref{algebraic-answer-1} and its `global' incarnation (Theorem \ref{geometric-global}). 
\end{rem} 

\begin{rem} \label{essence-1}
The non-trivial part of Theorem \ref{algebraic-answer-1} is the condition `only one place at $E \cup \tilde L$'. More precisely, removing this condition from Theorem \ref{algebraic-answer-1} yields the statement: ``$\tilde Y$ is algebraic iff there is a compact algebraic curve $C \subseteq Y$ such that $C$ does not intersect $E'$'' which is not hard to show. Indeed, here we sketch a proof. If $\tilde Y$ is algebraic, then there exists a compact algebraic curve $C \subseteq \tilde Y$ which does not pass through $P := \tilde \pi(\tilde E) \in \tilde Y$, which implies that $\tilde \pi^{-1}(C)$ does not intersect $\tilde E \supseteq E'$. For the opposite implication, consider the surface $Y'$ obtained from $Y$ by contracting all the components of $E$ other than $E^*$ (the contraction is possible due to Grauert's criterion). The singularities of $Y'$ are {\em sandwiched}, since there is a morphism $Y' \to \pp^2$. Since sandwiched singularities are rational \cite[Proposition 1.2]{lipman}, a criterion of Artin \cite{artractability} implies that $Y'$ is projective. Let $C$ be a closed algebraic curve on $Y$ which does not intersect $E'$ and let $C'$ (resp.\ $L'$, $E'^*$) be the image of $C$ (resp.\ $\tilde L$, $E^*$) on $Y'$. Then $C'$ is linearly equivalent (as a $\qq$-Cartier divisor) to $rL' + r_*E'^*$ for some $r, r_* \in \qq_{>0}$ and therefore a theorem of Zariski-Fujita \cite[Remark 2.1.32]{lazativityI} implies that for some $m \geq 1$, the line-bundle $\sheaf_{Y'}(mC')$ is base-point free. Let $\tilde Y'$ be the image of the morphism defined by sections of $\sheaf_{Y'}(mC')$. Since $C'$ does not intersect $L'$, it follows that $L'$ maps to a point in $\tilde Y'$, and therefore $\tilde Y \cong \tilde Y'$. Consequently, $\tilde Y$ is {\em projective}, and in particular, algebraic. 
\end{rem}  

\subsection{Effective criterion for algebraicity (in a simple case)} \label{subsec-effective}

In this subsection we state the {\em effective} version of Theorem \ref{algebraic-answer-1} in the simplest case (Theorem \ref{local-algebraic-answer}). We start with a discussion of a way to encode the input data of Question \ref{local-algebraic-question} in terms of a germ of a curve (and a positive integer). \\

We continue to use the notations of Subsection \ref{subsec-algebraic-introduction}. At first note that in the set up of Question \ref{local-algebraic-question} we may w.l.o.g.\ assume the following 
\begin{enumerate}
\item $\pi$ is a sequence of blow-ups such that every blow-up (other than the first one) is centered at a point on the exceptional divisor of the preceding blow-up. 
\item $E^*$ is the exceptional divisor of the {\em last} blow-up.
\end{enumerate}
Now assume the above conditions are satisfied. Let
\begin{description}
\item[] $\tilde C:= $ an analytic curve germ at a generic point on $E^*$ which is transversal to $E^*$,
\item[] $C := \pi(\tilde C)$, 
\item[] $r := $ (number of total blow-ups in $\pi$) $-$ (the minimum number of blow-ups after which the strict transform of $C$ transversally intersects the union of the strict transform of $L$ and the exceptional divisor).
\end{description}
It is straightforward to see that $L$, $\tilde C$ and $r$ uniquely determine $Y$, $E^*$ and $\tilde E$ via the following construction:

\paragraph{Construction of $Y$, $E^*$ and $\tilde E$ from $(L,C,r)$:}
\mbox{}
\begin{description}
\item[] $Y :=$ the surface formed by at first constructing (via a sequence of blow-ups) the minimal resolution of the singularity of $C \cup L$ and then blowing up the point of intersection of the strict transform of $C$ and the exceptional divisor $r$ more times,
\item[] $E^* :=$ the `last' exceptional divisor, i.e.\ the exceptional divisor of the last of the sequence of blow-ups in the construction of $Y^*$,
\item[] $\tilde E :=$ the union of the strict transform $\tilde L$ on $Y$ of $L$ with the strict transforms of all the exceptional divisors (of the sequence of blow-ups in the construction of $Y$) except $E^*$. 
\end{description}

\begin{figure}[htp]
		\begin{center}
			\begin{tikzpicture}[font = \scriptsize]
			\pgfmathsetmacro\scalefactor{1.1}
			\pgfmathsetmacro\centerdist{3*\scalefactor}

			\pgfmathsetmacro\distop{3*\scalefactor}
			\pgfmathsetmacro\arrowhordist{1.85*\scalefactor}

			\pgfmathsetmacro\arrowverttopdist{1.85*\scalefactor}
			\pgfmathsetmacro\arrowvertbottomdist{1.15*\scalefactor}

			\begin{scope}[shift={(-\centerdist,0)}, scale=\scalefactor]

				\draw[] (0,0) circle (1cm);
				\draw [green,thick] (-0.75, 0.2) -- (0.75, 0.2);

				\fill [red] (-.4,0.2) circle (1.25pt);
				\draw [red] (-0.4,0.375) node {$O$};
				
				\draw [blue, thick] plot [smooth] coordinates {(-0.3,-0.2) (-.55, 0.125) (-.4, 0.2) };
				\draw [blue, thick] plot [smooth] coordinates {(-0.5,-0.2) (-.65, 0.1) (-.4, 0.2) };

				\draw (0.8,0) node {\textcolor{green}{$L$} $:=$ a line};
				\draw (-0.8,-0.4) node [right] {\textcolor{blue}{$C:=\!$} analytically};
				\draw (-0.8,-0.6) node [right] {\phantom{\textcolor{blue}{$C:=\!$}}  irr.\ curve germ};
				\draw (0, -1.2) node {$\cc\pp^2$};

			\end{scope}
			
			\begin{scope}[shift={(-\centerdist, \distop)}, scale=\scalefactor]

				\draw[] (0,0) circle (1cm);
				\draw [green, thick] plot [smooth] coordinates {(-0.4,0.7) (-0.8,-0.4)};
				\draw (-0.5,0) node [green]{$\tilde L$};
				
				\draw [] (-0.7, 0.5) -- (0.3,0.8);
				\draw [blue,thick] plot [smooth] coordinates {(0, 0.2) (-0.1,0.4) (0,0.5) (-0.2, 0.75)};
				
				\draw (0.1,0.1) node [blue]{$C'$};

				\draw [] (-0.9, 0.1) -- (-0.4,-0.5);
				
			\end{scope}
			
			\begin{scope}[shift={(\centerdist, \distop)}, scale=\scalefactor]

				\draw[] (0,0) circle (1cm);
				\draw [green, thick] plot [smooth] coordinates {(-0.4,0.7) (-0.8,-0.4)};
				\draw (-0.45,0) node [green]{$\tilde L$};
				\draw [] (-0.7, 0.5) -- (0.3,0.8);
				\draw [] (-0.1, 0.8) -- (0.7, 0.4);
				\draw [] (-0.9, 0.1) -- (-0.4,-0.5);
				\draw [red, thick] (0.5, 0.6) -- (0.7, -0.1);
				\draw (0.5,-0.4) node [right, text width = 3cm] {\textcolor{red}{$E^*$} $:=$ `last' exceptional divisor}; 
				
				\draw [blue,thick] plot [smooth] coordinates {(0.8, 0.1) (0.5,0) (0.6,-0.1) (0.4, -0.2)};
				\draw (0.25,-0.2) node [blue]{$C'$};

				\draw [dashed] (-1, -0.6) -- (-1, 0.9) -- (0.8, 0.9) -- (0.8, 0.3) -- (-0.3,0.3) -- (-0.3, -0.6) -- cycle;
				\draw (0.8,0.6) node [right, text width = 4cm] {$\tilde E :=$ union of $\tilde L$ and all exceptional divisors except $E^*$}; 
			\end{scope}

			\begin{scope}[shift={(\centerdist, 0)}, scale=\scalefactor]

				\draw[] (0,0) circle (1cm);
				\draw [red, thick] plot [smooth] coordinates {(0.5, 0.6) (0.7, -0.1)};
				
				\draw [blue,thick] plot [smooth] coordinates {(0.8, 0.1) (0.5,0) (0.6,-0.1) (0.4, -0.2)};
				\draw (0.25,-0.2) node [blue]{$\tilde C$};		

				\draw (0, -1.2) node {$\tilde Y$}; 
				\draw (0,-1.5) node [shift={(2,0)}] {\phantom{f $\cc^2$ with one irr.\ curve at $\infty$}};
				
			\end{scope}

			\draw[->] (\arrowhordist, \distop) -- (-\arrowhordist,\distop) node [above, pos=0.5, text width = 3cm]{blow up $C' \cap $ (excep.\ divisors) $r$ more times};
			\draw[->] (-\centerdist, \arrowverttopdist) -- (-\centerdist, \arrowvertbottomdist) node [left, pos=0.4, text width=2.5cm, shift={(0,0)}] {resolution of singularities of $C \cup L$};

			\draw[->] (\centerdist, \arrowverttopdist) -- (\centerdist, \arrowvertbottomdist) node [right, pos=0.4, text width=3cm, shift={(0.2,0)}] {contract $\tilde E$ analytically (if possible)};
			
			\draw (0.85,-1.5) node [text width = 4cm] {Question: Is $\tilde Y$ algebraic?}; 

			\end{tikzpicture}
			\caption{Formulation of Question \ref{local-algebraic-question} in terms of $(L,C,r)$}\label{fig:L-c-r}
		\end{center}
\end{figure}
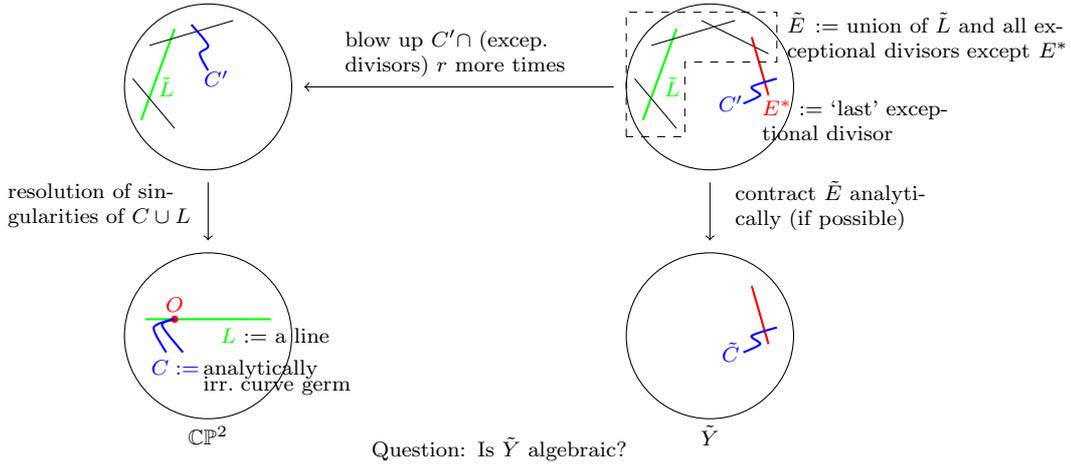

It follows that Questions \ref{local-existential-question} and \ref{local-algebraic-question} can be reformulated as below (see Figure \ref{fig:L-c-r}). 

\begin{bold-question} \label{local-2-question}
Let $L \subseteq \pp^2$ be a line, $C$ be an analytic curve-germ at a point $O \in L$ and $r$ be a non-negative integer. Let $Y_{L,C,r}$, $E^*_{L,C,r}$ and $\tilde E_{L,C,r}$ be the corresponding surface and divisors resulting from the above construction.
\begin{enumerate}
\item When is $\tilde E_{L,C,r}$ contractible?
\item When is $\tilde E_{L,C,r}$ algebraically contractible?
\end{enumerate}
\end{bold-question}

We now set up the notations for our answer to Question \ref{local-2-question}. Let $(u,v)$ be a system of {\em affine coordinates} (i.e.\ $u=0$ and $v=0$ are lines on $\pp^2$) at $O$ such that $L = \{u = 0\}$ and $O = (0,0)$. Let $v = \phi(u)$ be a Puiseux series expansion for $C$ at $O$. We start with a simple observation:

\begin{lemma}
If $\ord_u(\phi) \geq 1$, then $\tilde E_{L,C,r}$ is {\em not} contractible.
\end{lemma}

\begin{proof}
Indeed, $\ord_u(\phi) \geq 1$ implies that $C$ is {\em not} tangent to $L$, so that the strict transforms of $L$ and $C$ on the blow-up of $\pp^2$ at $O$ do not intersect. It follows that $\tilde L$ has self-intersection $\geq 0$, and consequently, is not contractible. 
\end{proof}

From now on we assume that $\ord_u(\phi) < 1$. Let the {\em Puiseux pairs} (see Definition \ref{puiseuxnition}) of $\phi$ be $(q_1, p_1), \ldots, (q_l, p_l)$ (note that $\ord_u(\phi)< 1$ implies that $l \geq 1$ and $\ord_u(\phi) =  q_1/p_1$). For every $\omega \in \rr$, let us write $[\phi]_{< \omega}$ for the (finite) Puiseux series obtained by summing up all terms of $\phi$ which have order $< \omega$. Define
\begin{align}
\alpha_{L,C,r} 	&:= \text{intersection multiplicity at $O$ of $C$ and a curve-germ with Puiseux expansion} \notag \\
				& \mbox{\phantom{$:\mathrel{=}\ $} } v = [\phi(u)]_{< (q_l+r)/p}  + \xi^* u^{(q_l+r)/p} + \hot\ \text{for a generic}\ \xi^* \in \cc \notag \\
				&= p \left( (p_1 \cdots p_l - p_2 \cdots p_l) \frac{q_1}{p_1} + (p_2 \cdots p_l - p_3 \cdots p_l) \frac{q_2}{p_1p_2} \right. \notag \\
				& \quad \qquad \left. + \cdots + (p_{l-1} p_{l} - p_{l}) \frac{q_{l-1}}{p_1 \cdots p_{l-1}}  + (p_l -1) \frac{q_l}{p_1 \cdots p_l} \right) + q_l + r,\quad \text{where} \label{local-alpha} \\
p 				&:= \text{polydromy order of $\phi$ (Definition \ref{puiseuxnition})} = p_1p_2 \cdots p_l. \notag			
\end{align} 

Grauert's criterion for contractibility translates (after some work) into the following in the set up of Question \ref{local-2-question}. This is an immediate corollary of \cite[Corollary 4.11 and Remark-Definition 4.13]{sub2-1}. 

\begin{prop} \label{local-existential-answer}
$\tilde E_{L,C,r}$ is contractible iff $\ord_u(\phi) < 0$ and $\alpha_{L,C,r} < p^2$.
\end{prop}

Now we give our criterion for algebraic contractibility in the case that $C$ has only one Puiseux pair, i.e.\ $l = 1$.

\begin{thm}[Effective criterion for algebraic contractibility when $l=1$] \label{local-algebraic-answer} 
Let $(L,C,r)$ be as in Question \ref{local-2-question}. Assume that the Puiseux expansion $v = \phi(u)$ of $C$ at $O$ has only one Puiseux pair $(q,p)$. Let $\omega$ be the weighted order on $\cc(u,v)$ which gives weights $p$ to $u$ and $q$ to $v$. Let $f(u,v)$ be the (unique) Weirstrass polynomial in $v$ which defines $C$ near $O$. Define $\tilde f$ to be the sum of all monomial terms of $f$ which have $\omega$-value less than $\alpha_{L,C,r} = pq+r$. Then $\tilde E_{L,C,r}$ is algebraically contractible iff it is contractible and $\deg_{(u,v)}(\tilde f) \leq p_1$, where $\deg_{(u,v)}$ is the usual degree in $(u,v)$-coordinates.
\end{thm}

We prove Theorem \ref{local-algebraic-answer} in Subsection \ref{sec-idea}.

\begin{example}[Continuation of Example \ref{non-example} - see also Remark \ref{1-puiseux-remark}] \label{non-example-again}
Let $L$ and $C_1$ and $C_2$ be as in Example \ref{non-example}. We consider Question \ref{local-2-question} for $C_1$ and $C_2$ and $r \geq 0$ (Example \ref{non-example} considered the case $r=8$). Figure \ref{fig:non-example-again} depicts the dual graph $\tilde E_{L,C_i,r}$; in particular $\tilde E_{L,C_i,r}$ is disconnected for $r = 0$.\\

\begin{figure}[htp]
\begin{center}

\subfigure[Case $r = 0$]{
\begin{tikzpicture}[scale=1, font = \small] 	
 	\pgfmathsetmacro\edge{.75}
 	\pgfmathsetmacro\vedge{.5}
	
 	\draw[thick] (-2*\edge,0) -- (-\edge,0);
 	\draw[thick] (0,-\vedge) -- (0,-2*\vedge);

 	\fill[black] ( - 2*\edge, 0) circle (3pt);
 	\fill[black] (-\edge, 0) circle (3pt);
 	\fill[black] (0, -\vedge) circle (3pt);
 	\fill[black] (0, - 2*\vedge) circle (3pt);
 	
 	\draw (-2*\edge,0)  node (e0up) [above] {$-1$};
 	\draw (-2*\edge,0 )  node (e0down) [below] {$\tilde L$};
 	\draw (-\edge,0 )  node (e1up) [above] {$-3$};	
 	\draw (0,-\vedge )  node (down1) [right] {$-2$};
 	\draw (0, -2*\vedge)  node (down2) [right] {$-3$};
 			 	
\end{tikzpicture}
}
\subfigure[Case $r \geq 1$]{
\begin{tikzpicture}[scale=1, font = \small] 	
 	\pgfmathsetmacro\dashedge{4.5}	
 	\pgfmathsetmacro\edge{.75}
 	\pgfmathsetmacro\vedge{.5}

 	\draw[thick] (-2*\edge,0) -- (\edge,0);
 	\draw[thick] (0,0) -- (0,-2*\vedge);
 	\draw[thick, dashed] (\edge,0) -- (\edge + \dashedge,0);
 	
 	\fill[black] ( - 2*\edge, 0) circle (3pt);
 	\fill[black] (-\edge, 0) circle (3pt);
 	\fill[black] (0, 0) circle (3pt);
 	\fill[black] (0, -\vedge) circle (3pt);
 	\fill[black] (0, - 2*\vedge) circle (3pt);
 	\fill[black] (\edge, 0) circle (3pt);
 	\fill[black] (\edge + \dashedge, 0) circle (3pt);
 	
 	\draw (-2*\edge,0)  node (e0up) [above] {$-1$};
 	\draw (-2*\edge,0 )  node (e0down) [below] {$\tilde L$};
 	\draw (-\edge,0 )  node (e1up) [above] {$-3$};
 	\draw (0,0 )  node (middleup) [above] {$-2$};	
 	\draw (0,-\vedge )  node (down1) [right] {$-2$};
 	\draw (0, -2*\vedge)  node (down2) [right] {$-3$};
 	\draw (\edge,0)  node (e+1-up) [above] {$-2$};
 	\draw (\edge + \dashedge,0)  node (e-last-1-up) [above] {$-2$};
 	
 	\draw [thick, decoration={brace, mirror, raise=5pt},decorate] (\edge,0) -- (\edge + \dashedge,0);
 	\draw (\edge + 0.5*\dashedge,-0.5) node [text width= 5cm, align = center] (extranodes) {$r-1$ vertices of weight $-2$};
 			 	
\end{tikzpicture}
}
\caption{Dual graph of $\tilde E_{L,C_i,r}$}\label{fig:non-example-again}
\end{center}
\end{figure}
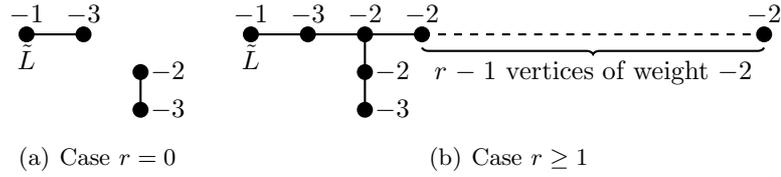

Recall that $C_i$'s are defined by $f_i = 0$, with $f_1 := v^5 - u^3$ and $f_2 := (v-u^2)^5 - u^3$. It follows that the Puiseux expansions in $u$ for each $C_i$ has only one Puiseux pair, namely $(3, 5)$. Moreover, each $f_i$ is a Weirstrass polynomial in $v$, so that we can use Theorems \ref{local-existential-answer} and \ref{local-algebraic-answer} to determine contractibility and algebraic contractibility of $\tilde E_{L,C_i,r}$.  \\

Identity \eqref{local-alpha} implies that $\alpha_{L,C_i,r} = pq + r = r + 15$ for each $i = 1,2$, and therefore Theorem \ref{local-existential-answer} implies that $\tilde E_{L, C_i,r}$'s are contractible iff $r < p^2 - pq = 10$. We now determine if the contractions are algebraic. The weighted degree $\omega$ Theorem \ref{local-algebraic-answer} is the same for both $i$'s, and it corresponds to weights $5$ for $u$ and $3$ for $v$. The $\tilde f$ of Theorem \ref{local-algebraic-answer} (computed from $f_i$'s) are as follows:
\begin{align*}
\tilde f_1 &= \begin{cases}
			0 & \text{if}\ r = 0,\\			
			v^5 - u^3 & \text{if}\ r\geq 1.
			\end{cases}
		&& &
\tilde f_2 &= \begin{cases}
			0 & \text{if}\ r = 0,\\			
			v^5 - u^3 & \text{if}\ 1 \leq r\leq 7, \\
			v^5 - u^3 - 5v^4u^2 & \text{if}\ 8 \leq r\leq 9.
			\end{cases}
\end{align*}			
Theorem \ref{local-algebraic-answer} therefore implies that $\tilde E_{L, C_1,r}$ is algebraically contractible for all $r < 10$, but $\tilde E_{L,C_2,r}$ is algebraically contractible only for $r \leq 7$. In particular, for $r = 8,9$, the contraction of $\tilde E_{L,C_2,r}$ produces a normal non-algebraic analytic surface. 
\end{example}

\subsection{The semigroup conditions on the sequence of {\em virtual poles}} \label{subsec-semigroup}

In this subsection we define the sequence of `virtual poles' corresponding to a curve-germ and state two `semigroup conditions' on these sequences. For a given singularity type (and a given $r$), if the virtual poles satisfy the first semigroup condition, this implies the existence of a curve-germ $C$ (with the prescribed singularity type) such that $\tilde E_{L,C,r}$ is algebraically contractible. On the other hand, satisfying {\em both} semigroup conditions ensures that $\tilde E_{L,C,r}$ are algebraically contractible for {\em all} curves $C$ with the given singularity type. The first semigroup condition is precisely the classical semigroup condition satisfied by generators of the semigroup of poles of a plane curve with one place at infinity.\\ 

We continue to use the notations of the set-up of Subsection \ref{subsec-effective}; in particular, we assume that the Puiseux expansion for $C$ is $v = \phi(u)$ with {\em Puiseux pairs} (Definition \ref{puiseuxnition}) $(q_1, p_1), \ldots, (q_l,p_l)$ with $l \geq 1$. Define $C_0 := L =  \{u = 0\}$, and for each $k$, $1 \leq k \leq l$, let $C_k$ be the curve-germ at $O$ with the Puiseux expansion $v = \phi_k(u)$, where $\phi_k(u)$ is the Puiseux series (with finitely many terms) consisting of all the terms of $\phi$ upto, but not including, the $k$-th characteristic exponent. Then it is a standard result (see e.g.\ \cite[Lemma 5.8.1]{casas-alvero}) that $m_k := (C, C_k)_O$, $0 \leq k \leq l$, are generators of the semigroup $\{(C, D)_O\}$ of intersection numbers at $O$, where $D$ varies among analytic curve-germs at $O$ not containing $C$. It follows from a straightforward computation that
\begin{subequations} \label{m-k-defn}
\begin{align}
m_0 &= p_1 \cdots p_l,\quad m_1 = q_1p_2 \cdots p_l,\ \text{and} \\
m_{k} &= p \left( (p_1 \cdots p_{k-1} - p_2 \cdots p_{k-1}) \frac{q_1}{p_1} + (p_2 \cdots p_{k-1} - p_3 \cdots p_{k-1}) \frac{q_2}{p_1p_2} \right. \notag  \\
		& \quad \qquad \left. + \cdots + (p_{k-1} - 1) \frac{q_{k-1}}{p_1 \cdots p_{k-1}}  + \frac{q_k}{p_1 \cdots p_k} \right),\ 2 \leq k \leq l. \label{m-k-general-defn}
\end{align}
\end{subequations} 

\begin{defn}[Virtual poles]\label{virtual-defn}
Let
\begin{align*}
\tilde l:= \begin{cases}
				l - 1 	& \text{if}\ r = 0, \\
				l		& \text{if}\ r > 0.
			\end{cases}
\end{align*}
The sequence of {\em virtual poles} at $O$ on $C$ are $\tilde m_0, \ldots, \tilde m_{\tilde l}$ defined as
\begin{subequations} \label{tilde-m-k-defn}
\begin{align}
\tilde m_0 &:= m_0,\quad 
\tilde m_1 := p_1 \cdots p_l - m_1,\quad
\tilde m_k := p_1^2 \cdots p_{k-1}^2 p_k \cdots p_l - m_k,\ 2 \leq k \leq \tilde l. \\
\intertext{The {\em generic virtual pole} at $O$ is}
\tilde m_{\tilde l+1} &:= \begin{cases}
							p_1^2 \cdots p_{l-1}^2 p_l - m_l = \frac{1}{p_l}\left(p^2 - \alpha_{L,C,r}\right)	& \text{if}\ r = 0, \\
							p_1^2 \cdots p_{l-1}^2 p_l^2 - p_lm_l - r = p^2 - \alpha_{L,C,r} 				& \text{if}\ r > 0.
						  \end{cases} 
\end{align}
\end{subequations}
\end{defn}

Fix $k$, $1 \leq k \leq \tilde l$. The {\em semigroup conditions} for $k$ are:
\begin{gather}
 p_k \tilde m_k \in \zz_{\geq 0} \langle \tilde m_0, \ldots, \tilde m_{k-1}\rangle .  \tag{S1-k} \label{semigroup-criterion-1} \\
(\tilde m_{k+1}, p_k \tilde m_k) \cap \zz\langle \tilde m_0, \ldots, \tilde m_k \rangle = (\tilde m_{k+1}, p_k \tilde m_k) \cap \zz_{\geq 0}\langle \tilde m_0, \ldots, \tilde m_k \rangle,\ \tag{S2-k} \label{semigroup-criterion-2}
\end{gather}
where $(\tilde m_{k+1}, p_k \tilde m_k): = \{a \in \rr: \tilde m_{k+1} < a < p_k \tilde m_k\}$ and $\zz_{\geq 0}\langle \tilde m_0, \ldots, \tilde m_k \rangle$ (respectively, $\zz \langle \tilde m_0, \ldots, \tilde m_k \rangle$) denotes the semigroup (respectively, group) generated by linear combinations of $\tilde m_0, \ldots, \tilde m_k$ with non-negative integer (respectively, integer) coefficients. 
\begin{thm} \label{semigroup-prop}
Let $(q_1, p_1), \ldots, (q_l,p_l)$ be pairs of relatively prime positive integers with $p_k \geq 2$, $1 \leq k \leq l$, and $r$ be a non-negative integer. Let $\tilde l$ and $\tilde m_0, \ldots, \tilde m_{\tilde l+1}$ be as in Definition \ref{virtual-defn}. Assume $\tilde m_{\tilde l +1} > 0$ (so that $\tilde E_{L,C,r}$ is contractible for every curve $C$ with Puiseux pairs $(q_1, p_1), \ldots, (q_l,p_l)$). Then
\begin{enumerate}
\item \label{algebraic-existence} There exists a curve-germ $C$ at $O$ with Puiseux pairs $(q_1, p_1), \ldots, (q_l,p_l)$ (for its Puiseux expansion $v = \phi(u)$) such that $\tilde E_{L,C,r}$ is algebraically contractible, iff the semigroup condition \eqref{semigroup-criterion-1} holds for all $k$, $1 \leq k \leq \tilde l$.
\item \label{non-algebraic-existence} There exists a curve-germ $C$ at $O$ with Puiseux pairs $(q_1, p_1), \ldots, (q_l,p_l)$ (for its Puiseux expansion $v = \phi(u)$) such that $\tilde E_{L,C,r}$ is {\em not} algebraically contractible, iff either \eqref{semigroup-criterion-1} or \eqref{semigroup-criterion-2} {\em fails} for some $k$, $1 \leq k \leq \tilde l$.
\end{enumerate}
\end{thm}

We prove the theorem in Section \ref{sec-idea} assuming the general case of the `effective criterion' (Theorem \ref{algebraic-global}).

\begin{rem}[`Explanation' of the term `virtual poles'] \label{curve-semigroup}
Let all notations be as in Theorem \ref{semigroup-prop}. In the set up of Question \ref{local-2-question}, identify $\pp^2 \setminus L$ with $\cc^2$, so that $(1/u, v/u)$ is a system of coordinates on $\cc^2$. The terminology `virtual poles' for $\tilde m_0, \ldots, \tilde m_{\tilde l}$ is motivated by the last assertion of the following result which is a reformulation of a fundamental result of the theory of plane algebraic curves with one place at infinity. 
\end{rem}

\begin{thm} [{\cite{abhya-moh-tschirnhausen}, \cite{abhyankar-expansion}, \cite{abhyankar-semigroup}, \cite{sathaye-stenerson}}] \label{curve-thm}
The semigroup condition \eqref{semigroup-criterion-1} is satisfied for all $k$, $1 \leq k \leq \tilde l$, iff there exists a curve $\tilde C$ in $\cc^2$ such that $\tilde C$ has only one place at infinity and has a Puiseux expansion at the point at infinity with Puiseux pairs $(q_1, p_1), \ldots, (q_{\tilde l}, p_{\tilde l})$. Moreover, if $\tilde C$ exists, then $\tilde m_0/\tilde p, \ldots, \tilde m_{\tilde l}/\tilde p$ are the generators of the {\em semigroup of poles at infinity} on $\tilde C$, where 
\begin{align*}
\tilde p:= \begin{cases}
				p_l 	& \text{if}\ \tilde l = l-1, \\
				1		& \text{if}\ \tilde l > l.
			\end{cases}
\end{align*}
\end{thm}

In the situation of \ref{curve-thm}, the numbers $\tilde m_k$, $0 \leq k \leq \tilde l$, are usually denoted in the literature by $\delta_k$, $0 \leq k \leq \tilde l$, and are called the {\em $\delta$-sequence} of $\tilde C$.\\

For positive integers $q,p$, and a curve-germ $C$ at $O$, we say that $C$ is of $(q,p)$-type with respect to $(u,v)$-coordinates iff $C$ has a Puiseux expansion $v = \phi(u)$ such that $(q,p)$ is the only Puiseux pair of $\phi$. The following result is a straightforward corollary of Theorem \ref{semigroup-prop} and the fact (which is a special case of \cite[Proposition 2.1]{herzog}) that the greatest integer not belonging to $\zz_{\geq 0}\langle p, p-q \rangle$ is $ p(p-q) - p - (p-q)$. 

\begin{cor} \label{cor-l=1}
Let $p,q$ be positive relatively prime integers and $r$ be a non-negative integer. 
\begin{enumerate}
\item \label{contractibly-q-p} Let $C$ be a $(q,p)$-type curve germ at $O$ with respect to $(u,v)$-coordinates. Then $\tilde E_{L,C,r}$ is contractible iff $r < p(p-q)$.
\item \label{algebraicity-q-p} There is a $(q,p)$-type curve germ $C$ at $O$ with respect to $(u,v)$-coordinates such that $\tilde E_{L,C,r}$ is contractible, but {\em not} algebraically contractible, iff $2p -q < r < p(p-q)$. \qed
\end{enumerate}
\end{cor}

\begin{rem} \label{1-puiseux-remark}
In fact, if $2p -q < r < p(p-q)$, Theorem \ref{local-algebraic-answer} gives an easy recipe to construct a curve $C$ such that $E_{L,C,r}$ is contractible, but {\em not} algebraically contractible; e.g.\ the curve given by $(v - f(u))^p = u^q$ would suffice for any polynomial $f(u) \in \cc[u]$ such that the coefficient of $u^2$ in $f(u)$ is non-zero. In Examples \ref{non-example} and \ref{non-example-aagain} we considered the case $(q,p) = (3,5)$ and $f(u) = u^2$. 
\end{rem}

\begin{remexample}[Dual graphs arising from {\em only} non-algebraic contractions]\label{non-2-remexample}
Note that the `virtual poles' of Theorem \ref{semigroup-prop} depend only on the {\em singularity type} of $C \cup L$, i.e.\ Puiseux pairs $(q_1, p_1), \ldots, (q_l,p_l)$ of the Puiseux expansion of the given curve in $(u,v)$-coordinates. If $(q_1,p_1), (q_2, p_2)$ are pairs of relatively prime positive integers such that $p_1, p_2 \geq 2$, $q_1 < p_1$ and 
\begin{align} \label{non-2-eqn}
q_2 = (p_1-q_1)(p_2-1)(p_1-1) + p_1(p_2+1),
\end{align}
then the `fact' stated preceding Corollary \ref{cor-l=1} implies that the condition \eqref{semigroup-criterion-1} fails for $k = 2$ and therefore Theorem \ref{semigroup-prop} implies that the dual graph for $E_{L,C,r}$ for $r = 1$ and any curve $C$ with Puiseux pairs $(q_1, p_1),(q_2,p_2)$ (for the Puiseux expansion in $u$) corresponds only to non-algebraic analytic contractions. Setting $(q_1,p_1) = (3,5)$ and $p_2 = 2$ in equation \eqref{non-2-eqn} gives $q_2 = 23$. Figure \ref{fig:non-2-example} depicts the dual graph of $\tilde E_{L,C,1}$ for a curve with Puiseux pairs $\{(3,5),(23,2)\}$ (for its Puiseux expansion in $u$).

\begin{figure}[htp]
\begin{center}
\begin{tikzpicture}[scale=1, font = \small] 	
 	\pgfmathsetmacro\dashedge{4.5}	
 	\pgfmathsetmacro\edge{.75}
 	\pgfmathsetmacro\vedge{.5}

 	\draw[thick] (-2*\edge,0) -- (\edge,0);
 	\draw[thick] (0,0) -- (0,-2*\vedge);
 	\draw[thick, dashed] (\edge,0) -- (\edge + \dashedge,0);
 	\draw[thick] (\edge + \dashedge,0) -- (3*\edge + \dashedge,0);
 	\draw[thick] (3*\edge + \dashedge,0) -- (3*\edge + \dashedge,-\vedge);
 	
 	\fill[black] ( - 2*\edge, 0) circle (3pt);
 	\fill[black] (-\edge, 0) circle (3pt);
 	\fill[black] (0, 0) circle (3pt);
 	\fill[black] (0, -\vedge) circle (3pt);
 	\fill[black] (0, - 2*\vedge) circle (3pt);
 	\fill[black] (\edge, 0) circle (3pt);
 	\fill[black] (\edge + \dashedge, 0) circle (3pt);
 	\fill[black] (2*\edge + \dashedge, 0) circle (3pt);
 	\fill[black] (3*\edge + \dashedge, 0) circle (3pt);
 	\fill[black] (3*\edge + \dashedge, -\vedge) circle (3pt);
 	
 	\draw (-2*\edge,0)  node (e0up) [above] {$-1$};
 	\draw (-2*\edge,0 )  node (e0down) [below] {$\tilde L$};
 	\draw (-\edge,0 )  node (e1up) [above] {$-3$};
 	\draw (0,0 )  node (middleup) [above] {$-2$};	
 	\draw (0,-\vedge )  node (down1) [right] {$-2$};
 	\draw (0, -2*\vedge)  node (down2) [right] {$-3$};
 	\draw (\edge,0)  node (e+1-up) [above] {$-2$};
 	\draw (\edge + \dashedge,0)  node (e-last-1-up) [above] {$-2$};
 	\draw (2*\edge + \dashedge,0)  node [above] {$-3$};
 	\draw (3*\edge + \dashedge,0)  node [above] {$-2$};
 	\draw (3*\edge + \dashedge,-\vedge)  node [right] {$-2$};
 	
 	\draw [thick, decoration={brace, mirror, raise=5pt},decorate] (\edge,0) -- (\edge + \dashedge,0);
 	\draw (\edge + 0.5*\dashedge,-0.5) node [text width= 5cm, align = center] (extranodes) {$7$ vertices of weight $-2$};
 			 	
\end{tikzpicture}

\caption{A dual graph of $\tilde E_{L,C,r}$ which comes from only non-algebraic analytic contractions}\label{fig:non-2-example}
\end{center}
\end{figure}
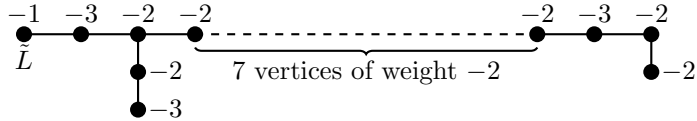

\end{remexample}

\section{The {\em global} incarnation of the question of algebraicity} \label{sec-global}
In the set up of Question \ref{local-algebraic-question}, identifying $\pp^2 \setminus L$ with $\cc^2$ and $\tilde Y$ with a {\em compactification} of $\cc^2$ translates Question \ref{local-algebraic-question} to the following

\begin{question} \label{global-question}
Let $\bar X$ be a normal analytic compactification of $X:= \cc^2$ such that $\bar X\setminus X$ is an irreducible curve. When is $\bar X$ algebraic?
\end{question}

In this section we give complete statements of geometric and algebraic (which is also effective!) answers to Question \ref{global-question}, and in Section \ref{sec-proof} we present a proof of these statements under an additional simplifying condition.

\subsection{Geometric answer}

Let $X := \cc^2$ and $\bar X^0 := \pp^2 \supseteq X$. Let $\bar X$ be a normal analytic compactification of $X$ such that $\bar X\setminus X$ is an irreducible curve and $\sigma': \bar X^0 \dashrightarrow \bar X$ be the bimeromorphic map induced by identification of $X$. Let $S'$ be the (finite) set of points of indeterminacies of $\sigma'$. 

\begin{thm} \label{geometric-global}
Assume $\sigma'$ is not an isomorphism, so that $\sigma'$ maps $L_\infty \setminus S'$ to a point $P_\infty \in C_\infty$. Then $\bar X$ is algebraic iff there is an algebraic curve $C \subseteq X$ with one place at infinity such that $\bar C^{\bar X} \cap P_\infty = \emptyset$, where $\bar C^{\bar X}$ is the closure of $C$ in $\bar X$. 
\end{thm}

\begin{figure}[htp]
		\begin{center}
			\begin{tikzpicture}[font = \scriptsize]
			\pgfmathsetmacro\scalefactor{1}
			\pgfmathsetmacro\centerdist{2*\scalefactor}

			\pgfmathsetmacro\distop{2.5*\scalefactor}
			\pgfmathsetmacro\arrowbelowdist{0.75*\scalefactor}

			\pgfmathsetmacro\arrowdiagtopydist{1.55*\scalefactor}
			\pgfmathsetmacro\arrowdiagtopxdist{1.25*\scalefactor}

			\pgfmathsetmacro\arrowdiagbottomydist{1*\scalefactor}
			\pgfmathsetmacro\arrowdiagbottomxdist{1.5*\scalefactor}

			\begin{scope}[shift={(\centerdist,0)}, scale=\scalefactor]

				\draw[] (0,0) circle (1cm);
				
				\draw [red, thick] plot [smooth] coordinates {(.25,0.45) (0,0.2) (-0.4,-0.1) (-0.5, -0.3) (-0.6, - 0.5) (-0.7,-0.6)};
				
				\fill [green] (0,0.2) circle (1.25pt);
				\draw [green] (0,0.2) node[right] {$P_\infty$};
				
				\begin{scope}[shift={(0.45,-0.25)}]
					\draw [blue,thick] plot [smooth] coordinates {(-0.85, 0.15) (-0.83,-0) (-0.5,0.2) (0.2, -0.05)};
					\draw (0.3,-0.15) node [blue]{$C$};
				\end{scope}

				\draw [red] (0.3,0.65) node {$C_\infty$};
				\draw (0,-1.2) node {$\bar X$};
			\end{scope}

			\begin{scope}[shift={(-\centerdist,0)}, scale=\scalefactor]

				\draw[] (0,0) circle (1cm);
				\draw [green,thick] (-0.75, 0.2) -- (0.75, 0.2);

				\draw [green] (0.4,0.375) node {$L_\infty$};
				\draw (0, -1.2) node {$\bar X^0$};
			\end{scope}
			
			\draw[dash pattern=on 2pt off 1pt,->] (-\arrowbelowdist, 0) -- (\arrowbelowdist,0) node [above, pos=0.5]{$\sigma'$};

			\end{tikzpicture}
			\caption{Geometric answer to Question \ref{global-question}} \label{fig:geom-global-answer}
		\end{center}
\end{figure}
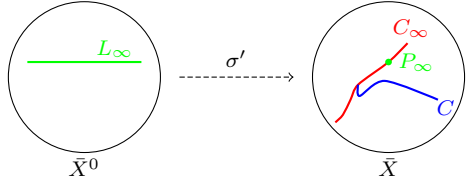 

\subsection{Algebraic answer}
As in the preceding subsection, let $\bar X$ be a normal analytic compactification of $X := \cc^2$ such that $C_\infty := \bar X\setminus X$ is an irreducible curve. Let $\nu: \cc(X)\setminus\{0\} \to \zz$ be the order of vanishing along $C_\infty$. Then $\nu$ is a {\em divisorial valuation on $\cc(X)$} (see Definition \ref{divisorial-defn}) which is {\em centered at infinity} (i.e.\ there are $f \in \cc[x,y] \setminus \{0\}$ such that $\nu(f) < 0$). We study $\bar X$ via studying $\nu$, or more precisely $\delta := -\nu$, which we call a {\em semidegree} (Definition \ref{sub-defn}). In Subsection \ref{key-section} below we associate with $\delta$ a (finite) sequence of elements of $\cc[x,x^{-1},y]$ which we call {\em key forms} (which are analogues of {\em key polynomials} \cite{maclane-key} associated to $\nu$). The algebraic formulation of our result is then:
\begin{thm} \label{algebraic-global}
$\bar X$ is algebraic iff all the key forms associated to $\delta$ are {\em polynomials} iff the {\em last} key form associated to $\delta$ is a polynomial.
\end{thm}
Below we recall the notion of key polynomials associated to valuations and then define key forms for semidegrees.

\subsection{Puiseux series and Key Polynomials corresponding to valuations} \label{pre-key-section}

\begin{defn}[Divisorial valuations] \label{divisorial-defn}
Let $u, v$ be polynomial coordinates on $X' \cong \cc^2$. A discrete valuation on $\cc(u,v)$ is a map $\nu: \cc(u,v)\setminus\{0\} \to \zz$, such that for all $f,g \in \cc(u,v)\setminus \{0\}$,
\begin{compactenum}
\item $\nu(f+g) \geq \min\{\nu(f), \nu(g)\}$,
\item $\nu(fg) = \nu(f) + \nu(g)$.
\end{compactenum}
Let $\bar X'$ be an algebraic compactification of $X'$. A discrete valuation $\nu$ on $\cc(u,v)$ is called {\em divisorial} iff there exists a normal algebraic surface $Y$ equipped with a birational morphism $\sigma: Y \to \bar X$ and a curve $C_\nu$ on $Y$ such that for all non-zero $f \in \cc[x,y]$, $\nu(f)$ is the order of vanishing of $\sigma^*(f)$ along $C_\nu$. The {\em center} of $\nu$ on $\bar X'$ is $\sigma(C_\nu)$. 
\end{defn}

Let $u,v$ be as in Definition \ref{divisorial-defn} and $\nu$ be a divisorial valuation on $\cc(u,v)$ with $\nu(u) >0$ and $\nu(v) > 0$. We recall two of the standard ways of representing a valuation: by a Puiseux series and by {\em key polynomials} \cite{maclane-key}. 

\begin{defn}[Puiseux series] \label{puiseuxnition}
Recall that the ring of Puiseux series in $u$ is 
$$\psxcu := \bigcup_{p=1}^\infty \cc[[u^{1/p}]] = \left\{\sum_{k=0}^\infty a_k u^{k/p} : p \in \zz,\ p \geq 1\right\}.$$
Let $\phi \in \psxcu$. The {\em polydromy order} \cite[Chapter 1]{casas-alvero} of $\phi$ is the smallest positive integer $p$ such that $\phi \in \cc[[u^{1/p}]]$. For any $r \in \qq$, let us denote by $[\phi]_{<r}$ (resp.\ $[\phi]_{\leq r}$) sum of all terms of $\phi$ with order less than (resp.\ less than or equal to) $r$. Then the {\em Puiseux pairs} of $\phi$ are the unique sequence of pairs of relatively prime positive integers $(q_1, p_1), \ldots, (q_k,p_k)$ such that the polydromy order of $\phi$ is $p_1\cdots p_k$, and for all $j$, $1 \leq j \leq k$,
\begin{compactenum}
\item $p_j \geq 2$,
\item $[\phi]_{<\frac{q_j}{p_1\cdots p_j}} \in \cc[u^{\frac{1}{p_0\cdots p_{j-1}}}]$ (where we set $p_0 := 1$), and 
\item $[\phi]_{\leq \frac{q_j}{p_1\cdots p_j}} \not\in \cc[u^{\frac{1}{p_0\cdots p_{j-1}}}]$.
\end{compactenum}
\end{defn}

\begin{prop}[Valuation via Puiseux series: cf.\ {\cite[Proposition 4.1]{favsson-tree}}] \label{valseux}
There exists a {\em Puiseux polynomial} (i.e.\ a Puiseux series with finitely many terms) $\phi_\nu \in \psxcu$ and a rational number $r_\nu$ such that for all $f \in \cc[u,v]$, 
\begin{align}
\nu(f) = \nu(u)\ord_u\left( f(u,v)|_{v = \phi_\nu(u) + \xi u^{r_\nu}}\right), \label{favsson-puiseux}
\end{align}
where $\xi$ is an indeterminate. 
\end{prop}

\begin{defn} \label{generic-puiseux-nition}
If $\phi_\nu$ and $r_\nu$ are as in Proposition \ref{valseux}, we say that $\tilde \phi_\nu(u,\xi):= \phi_\nu(x) + \xi u^{r_\nu}$ is the {\em generic Puiseux series} associated to $\nu$.
\end{defn}

\begin{defn}[{Key Polynomials of \cite{maclane-key} after \cite[Chapter 2]{favsson-tree}}]
Let $\nu$ be as above. A sequence of polynomials $U_0, U_1, \ldots, U_k \in \cc[u,v]$ is called the sequence of {\em key polynomials} for $\nu$ if the following properties are satisfied:
\begin{compactenum}
\addtocounter{enumi}{-1}
\item $U_0 = u$, $U_1 = v$.
\item \label{semigroup-val-property} Let $\omega_j := \nu(U_j)$, $0 \leq j \leq k$. Then 
\begin{align*}
\omega_{j+1} > n_j \omega_j = \sum_{i = 0}^{j-1}m_{j,i}\omega_i\ \text{for}\ 1 \leq j < k,
\end{align*}
where $n_j \in \zz_{> 0}$ and $m_{j,i} \in \zz_{\geq 0}$ satisfy
\begin{gather*}
n_j = \min\{l \in \zz_{> 0}; l\omega_j \in \zz \omega_0 + \cdots + \zz \omega_{j-1}\}\ \text{for}\ 1 \leq j < k,\ \text{and}\\
0 \leq m_{j,i} < n_i\ \text{for}\ 1 \leq i < j < k.
\end{gather*}

\item For $1 \leq j < k$, there exists $\theta_j \in \cc^*$ such that 
\begin{align*}
U_{j+1} = U_j^{n_j} - \theta_j U_0^{m_{j,0}} \cdots U_{j-1}^{m_{j,j-1}}.
\end{align*}

\item Let $u_0, \ldots, u_k$ be indeterminates and $\omega$ be the {\em weighted order} on $\cc[u_0, \ldots, u_k]$ corresponding to weights $\omega_j$ for $u_j$, $0 \leq j \leq k$ (i.e.\ the value of $\omega$ on a polynomial is the smallest `weight' of its monomials). Then for every polynomial $f \in \cc[u,v]$, 
\begin{align*}
\nu(f) = \max\{\omega(F): F \in \cc[u_0, \ldots, u_k],\ F(U_0, \ldots, U_k) = f\}.
\end{align*}
\end{compactenum} 
\end{defn}

\begin{thm}[{\cite[Theorem 2.29]{favsson-tree}}]
There is a unique and finite sequence of key polynomials for $\nu$. 
\end{thm}

\begin{example}
If $\nu$ is the multiplicity valuation at the origin, then the generic Puiseux series corresponding to $\nu$ is $\tilde \phi_\nu = \xi u$ and the key polynomials are $u,v$.
\end{example}

\begin{example}
If $\nu$ is the weighted order in $(u,v)$-coordinates corresponding to weights $p$ for $u$ and $q$ for $v$ with $p,q$ positive integers, then $\tilde \phi_\nu = \xi^{q/p}$ and the key polynomials are again $u,v$.
\end{example}

\begin{example} \label{general-example}
Let $C$ be a singular irreducible analytic curve-germ at the origin with Puiseux expansion $v = \phi(u)$. Pick any positive integer $r$. Construct the minimal resolution of singularity of $C$ (at $O$) and then blow up $r$ more times the point where the strict transform of $C$ intersects the exceptional divisor. Let $E$ be the {\em last} exceptional divisor constructed via this process and $\nu$ be the valuation corresponding to $E$. Then the generic Puiseux series corresponding to $\nu$ is 
\begin{align*} 
\tilde \phi_\nu &= [\phi(u)]_{< (q+r)/p} + \xi u^{(q+r)/p},\quad \text{where}\\
p &= \text{the smallest positive integer such that $\phi \in \cc[[u^{1/p}]]$,} \\
q/p &= \text{the {\em last} Puiseux exponent of $\phi$,}\\
[\phi(u)]_{< (q+r)/p} &= \text{sum of all terms of $\phi(u)$ with order less than $(q+r)/p$.}
\end{align*}
\end{example}

\begin{example} \label{non-example-aagain}
Let $C_1$ and $C_2$ be the curves from Example \ref{non-example-again}. We apply the construction of Example \ref{general-example} to $C_1$ and $C_2$. The Puiseux expansion for $C_1$ and $C_2$ at the origin are respectively given by: $v = u^{3/5}$ and $v= u^{3/5} + u^2$. It follows that the generic Puiseux series for the valuation of Example \ref{general-example} applied to $C_i$'s are:
\begin{align*}
\tilde \phi_{\nu_1} &= \begin{cases}
			\xi u^{3/5} & \text{if}\ r = 0,\\			
			u^{3/5} + \xi u^{(3+r)/5} & \text{if}\ r\geq 1.
			\end{cases}
		&& &
\tilde \phi_{\nu_2} &= \begin{cases}
			\xi u^{3/5} & \text{if}\ r = 0,\\			
			u^{3/5} + \xi u^{(3+r)/5} & \text{if}\ 1 \leq r\leq 7, \\
			u^{3/5} + u^2 + \xi u^{(3+r)/5} & \text{if}\ 8 \leq r.
			\end{cases}
\end{align*}		
The sequence of key polynomials for $\nu_1$ and $\nu_2$ for $0 \leq r < 10$ are as follows:
\begin{align*}
\parbox{2cm}{key polynomials for $\nu_1$} 
		 &= \begin{cases}
			u,v & \text{if}\ r = 0,\\			
			u,v, v^5 - u^3 & \text{if}\ r\geq 1.
			\end{cases}
		&& &
\parbox{2cm}{key polynomials for $\nu_2$}
		 &= \begin{cases}
			u,v & \text{if}\ r = 0,\\			
			u,v,v^5 - u^3 & \text{if}\ 1 \leq r\leq 7, \\
			u,v, v^5 - u^3, v^5 - u^3 - 5v^4u^2 & \text{if}\ 8 \leq r\leq 9.
			\end{cases}
\end{align*}			
In particular, note that for $r \geq 1$ the last key polynomials are precisely the $\tilde f_i$'s of Example \ref{non-example-again}. This is in fact the key observation for the proof of Theorem \ref{local-algebraic-answer} using Theorem \ref{algebraic-global}.
\end{example}

\subsection{Degree-wise Puiseux series and Key Forms corresponding to semidegrees} \label{key-section}

Let $X \cong \cc^2$ with coordinates $(x,y)$ and let $\delta$ be a {\em divisorial semidegree} (i.e.\ $\nu := -\delta$ is a divisorial valuation) on $\cc[x,y]$ such that $\delta(x) > 0$. Let $u := 1/x$ and $v := y/x^k$ for some $k$ such that $\delta(y) < k\delta(x)$. Then $\nu(u) >0$ and $\nu(v) > 0$. Applying Proposition \ref{valseux} to $\nu$ and then translating in terms of $(x,y)$-coordinates yields Proposition \ref{valdeg} below. Recall that the field $\cc((x))$ of Laurent series in $x$ is the field of fractions of the formal power series ring $\cc[[x]]$. The field of {\em degree-wise Puiseux series} in $x$ is 
$$\dpsxc := \bigcup_{p=1}^\infty \cc((x^{-1/p})) = \left\{\sum_{j \leq k} a_j x^{j/p} : k,p \in \zz,\ p \geq 1 \right\}.$$

\begin{prop}[{\cite[Theorem 1.2]{sub2-1}}] \label{valdeg}
There exists a {\em degree-wise Puiseux polynomial} (i.e.\ a degree-wise Puiseux series with finitely many terms) $\phi_\delta \in \dpsxc$ and a rational number $r_\delta < \ord_x(\phi_\delta)$ such that for every polynomial $f \in \cc[x,y]$, 
\begin{align}
\delta(f) = \delta(x)\deg_x\left( f(x,y)|_{y = \phi_\delta(x) + \xi x^{r_\delta}}\right), \label{phi-delta-defn}
\end{align}
where $\xi$ is an indeterminate. 
\end{prop}

\begin{defn}
If $\phi_\delta$ and $r_\delta$ are as in Proposition \ref{valdeg}, we say that $\tilde \phi_\delta(x,\xi):= \phi_\delta(x) + \xi x^{r_\delta}$ is the {\em generic degree-wise Puiseux series} associated to $\delta$.
\end{defn}

We will need the following geometric interpretation of degree-wise Puiseux series: assume that $\bar X$ is a normal analytic compactification of $X$ with an irreducible curve $C_\infty$ at infinity and $\delta$ is precisely the order of pole along $C_\infty$. Let $\bar X^0 \cong \pp^2$ be the compactification of $X$ induced by the map $(x,y) \mapsto [1:x:y]$, $\sigma: \bar X \dashrightarrow \bar X^0$ be the natural bimeromorphic map, and $S$ (resp.\ $S'$) be the finite set of points of indeterminacy of $\sigma$ (resp.\ $\sigma^{-1})$. Assume that $\sigma$ maps $C_\infty \setminus S$ to a point $O \in L_\infty := \bar X^0 \setminus X$. It then follows that $\sigma^{-1}$ maps $L_\infty \setminus S'$ to a point $P_\infty \in C_\infty$.

\begin{prop}[{\cite[Proposition 4.2]{sub2-1}}] \label{prop-param}
Let $\tilde \phi_\delta(x,\xi)$ be the generic degree-wise Puiseux series associated to $\delta$ and $\gamma$ be an (analytically) irreducible curve-germ at $O$ (on $\bar X^0$) which is distinct from the germ of $L_\infty$. Then the strict transform of $\gamma$ on $\bar X$ intersects $C_\infty \setminus \{P_\infty\}$ iff $\gamma \cap X$ (i.e.\ the {\em finite} part of $\gamma$) has a parametrization of the form 
\begin{align}
t \mapsto (t, \tilde \phi_\delta(t,\xi)|_{\xi = c} + \lot)\quad \text{for}\ |t| \gg 0 \label{dpuiseux-param} \tag{$*$}
\end{align}
for some $c \in \cc$, where $\lot$ means `lower order terms' (in $t$).
\end{prop}

Now we adapt the notion of key polynomials to the case of semidegrees. The main difference from the case of valuations is that these may {\em not} be polynomials (hence the word `form'\footnote{We use the word `form' in particular, because the key forms have a property analogous to `weighted homogeneous forms' for weighted degrees. Indeed, for each key form $f_j$ of $\delta$ there is a semidegree $\delta_j$ which is an {\em approximation} of $\delta$ such that $f_j$ is the {\em leading form} of an element in $\cc[x,x^{-1},y]$ with respect to $\delta_j$.} instead of `polynomial') - see Example \ref{non-example-aaagain} and Remark \ref{non-remark}.

\begin{defn}[Key Forms] \label{key-defn}
Let $\delta$ be as above. A sequence of elements $f_0, f_1, \ldots, f_k \in \cc[x,x^{-1},y]$ is called the sequence of {\em key forms} for $\delta$ if the following properties are satisfied:
\begin{compactenum}
\let\oldenumi\theenumi
\renewcommand{\theenumi}{P\oldenumi}
\addtocounter{enumi}{-1}
\item $f_0 = x$, $f_1 = y$.
\item \label{semigroup-property} Let $\omega_j := \delta(f_j)$, $0 \leq j \leq k$. Then 
\begin{align*}
\omega_{j+1} < n_j \omega_j = \sum_{i = 0}^{j-1}m_{j,i}\omega_i\ \text{for}\ 1 \leq j < k,
\end{align*}
where 
\begin{compactenum}
\item $n_j = \min\{l \in \zz_{> 0}; l\omega_j \in \zz \omega_0 + \cdots + \zz \omega_{j-1}\}$ for $1 \leq j < k$,
\item $m_{j,i}$'s are integers such that $0 \leq m_{j,i} < n_i$ for $1 \leq i < j < k$ (in particular, $m_{j,0}$'s are allowed to be {\em negative}). 
\end{compactenum}

\item \label{next-property} For $1 \leq j < k$, there exists $\theta_j \in \cc^*$ such that 
\begin{align*}
f_{j+1} = f_j^{n_j} - \theta_j f_0^{m_{j,0}} \cdots f_{j-1}^{m_{j,j-1}}.
\end{align*}

\item \label{generating-property} Let $y_0, \ldots, y_k$ be indeterminates and $\omega$ be the {\em weighted degree} on $\cc[y_0, \ldots, y_k]$ corresponding to weights $\omega_j$ for $y_j$, $0 \leq j \leq k$ (i.e.\ the value of $\eta$ on a polynomial is the maximum `weight' of its monomials). Then for every polynomial $f \in \cc[x,y]$, 
\begin{align*}
\nu(f) = \min\{\eta(F): F \in \cc[y_0, \ldots, y_k],\ F(f_0, \ldots, f_k) = f\}.
\end{align*}
\end{compactenum} 
\end{defn}

\begin{thm}
There is a unique and finite sequence of key forms for $\delta$. 
\end{thm}

\begin{example} \label{deg-dpuiseux-key}
If $\delta$ is a weighted degree in $(x,y)$-coordinates corresponding to weights $p$ for $x$ and $q$ for $y$ with $p,q$ positive integers, then the generic degree-wise Puiseux series corresponding to $\delta$ is $\tilde \phi_\delta = \xi^{q/p}$ and the key polynomials are $f_0 = x$ and $f_1 = y$.
\end{example}

\begin{example} \label{non-example-aaagain}
Set $u := 1/x$ and $v := y/x$. Let $\nu_1$ and $\nu_2$ be valuations from Example \ref{non-example-aagain} and set $\delta_i := -\nu_i$, $1 \leq i \leq 2$. It follows from the computations of Example \ref{non-example-aagain} shows that the generic degree-wise Puiseux series for the valuation of Example \ref{general-example} applied to $C_i$'s are:
\begin{align*}
\tilde \phi_{\nu_1} &= \begin{cases}
			\xi x^{2/5} & \text{if}\ r = 0,\\			
			x^{2/5} + \xi x^{(2-r)/5} & \text{if}\ r\geq 1.
			\end{cases}
		&& &
\tilde \phi_{\nu_2} &= \begin{cases}
			\xi x^{2/5} & \text{if}\ r = 0,\\			
			x^{2/5} + \xi x^{(2-r)/5} & \text{if}\ 1 \leq r\leq 7, \\
			x^{2/5} + x^2 + \xi x^{(2-r)/5} & \text{if}\ 8 \leq r.
			\end{cases}
\end{align*}		
The sequence of key polynomials for $\delta_1$ and $\delta_2$ for $0 \leq r < 10$ are as follows:
\begin{align*}
\parbox{2cm}{key polynomials for $\nu_1$} 
		 &= \begin{cases}
			x,y & \text{if}\ r = 0,\\			
			x,y, y^5 - x^2 & \text{if}\ r\geq 1.
			\end{cases}
		&& &
\parbox{2cm}{key polynomials for $\nu_2$}
		 &= \begin{cases}
			x,y & \text{if}\ r = 0,\\			
			x,y, y^5 - x^2 & \text{if}\ 1 \leq r\leq 7, \\
			x,y, y^5 - x^2, y^5 - x^2 - 5y^4x^{-1} & \text{if}\ 8 \leq r\leq 9.
			\end{cases}
\end{align*}			
In particular, for $8 \leq r \leq 9$, the last key polynomial for $\delta_2$ is {\em not} a polynomial. On the other hand, recall (from Example \ref{non-example-again}) that $\tilde E_{L,C_2,r}$ is contractible for these values of $r$, which implies that $\delta_2$ is {\em positive} on $\cc[x,y]\setminus\{0\}$. 
\end{example}

\begin{rem}\label{non-remark}
As Example \ref{non-example-aaagain} illustrates, even if $\delta$ is positive on $\cc[x,y]\setminus\{0\}$, some of the key forms may not be polynomials. This is precisely the reason of the difficulty of the global case and the `content' of the algebraicity criteria of this article is the statement that this does {\em not} happen if $\delta$ is the semidegree corresponding to the curve at infinity on an algebraic compactification of $\cc^2$ for which the curve that infinity is irreducible.
\end{rem} 

\section{Proof of the results in the case of one Puiseux pair} \label{sec-proof}
Let $\bar X$ be a normal analytic compactification of $X := \cc^2$ with $C_\infty := \bar X \setminus X$ irreducible and let $\delta$ be the semidegree on $\cc(x,y)$ corresponding to $C_\infty$. In this section we give a proof of Theorems \ref{geometric-global} and \ref{algebraic-global} under the additional assumption that the generic degree-wise Puiseux series for $\delta$ has at most one Puiseux pair. In the local setting this gives a complete proof of Theorem \ref{local-algebraic-answer}. We also give a proof of Theorem \ref{semigroup-prop}. At first we briefly recall some notions we use in the proof: the process of compactifications via {\em degree-like functions} and the factorization of polynomials in terms of degree-wise Puiseux series (the latter being just a reformulation of the factorization in terms of Puiseux series). 

\subsection{Background} \label{section-background}
\subsubsection{Degree-like functions and compactifications} \label{degree-like-section}
\begin{defn} \label{degree-like-defn}
Let $X$ be an irreducible affine variety over an algebraically closed field $\kk$. A map $\delta: \kk[X] \setminus \{0\} \to \zz$ is called a {\em degree-like function} if 
\begin{compactenum}
\item \label{deg1} $\delta(f+g) \leq \max\{\delta(f), \delta(g)\}$ for all $f, g \in \kk[X]$, with $<$ in the preceding inequality implying $\delta(f) = \delta(g)$.
\item \label{deg2} $\delta(fg) \leq \delta(f) + \delta(g)$ for all $f, g \in \kk[X]$.
\end{compactenum}
\end{defn}

Every degree-like function $\delta$ on $\kk[X]$ defines an {\em ascending filtration} $\scrF^\delta := \{F^\delta_d\}_{d \geq 0}$ on $\kk[X]$, where $F^\delta_d := \{f \in \kk[X]: \delta(f) \leq d\}$. Define
$$ \kk[X]^\delta := \dsum_{d \geq 0} F^\delta_d, \quad \gr  \kk[X]^\delta := \dsum_{d \geq 0} F^\delta_d/F^\delta_{d-1}.$$

\begin{rem}
For every $f \in \kk[X]$, there are infinitely many `copies' of $f$ in $\kk[X]^\delta$, namely the copy of $f$ in $F^\delta_d$ for {\em each $d \geq \delta(f)$}; we denote the copy of $f$ in $F^\delta_d$ by $(f)_d$. If $t$ is a new indeterminate, then 
$$\kk[X]^\delta \cong \sum_{d \geq 0} F^\delta_d t^d,$$
via the isomorphism $(f)_d \mapsto ft^d$. Note that $t$ corresponds to $(1)_1$ under this isomorphism.
\end{rem}

We say that $\delta$ is {\em finitely-generated} if $\kk[X]^\delta$ is a finitely generated algebra over $\kk$ and that $\delta$ is {\em projective} if in addition $F^\delta_0 = \kk$. The motivation for the terminology comes from the following straightforward

\begin{prop}[{\cite[Proposition 2.5]{sub1}}] \label{basic-prop}
If $\delta$ is a projective degree-like function, then $\xdelta := \proj \kk[X]^\delta$ is a projective compactification of $X$. The {\em hypersurface at infinity} $\xdelta_\infty := \xdelta \setminus X$ is the zero set of the $\qq$-Cartier divisor defined by $(1)_1$ and is isomorphic to $\proj \gr \kk[X]^\delta$. Conversely, if $\bar X$ is any projective compactification of $X$ such that $\bar X \setminus X$ is the support of an effective ample divisor, then there is a projective degree-like function $\delta$ on $\kk[X]$ such that $\xdelta \cong \bar X$.
\end{prop}

\begin{defn} \label{sub-defn}
A degree-like function $\delta$ is called a {\em semidegree} if it always satisfies property \ref{deg2} with an equality, and $\delta$ is called a {\em subdegree} if it is the maximum of finitely many semidegrees. As we have already seen in Section \ref{sec-global}, a semidegree is the negative of a {\em discrete valuation}. 
\end{defn}

\begin{thm}[{cf.\ \cite[Theorem 4.1]{sub1}}] \label{sub-structure}
Let $\delta$ be a finitely generated degree-like function on the coordinate ring of an irreducible affine variety $X$. Let $I$ be the ideal of $\kk[X]^\delta$ generated by $(1)_1$. Then 
\begin{enumerate}
\item $\delta$ is a semidegree (resp.\ subdegree) iff $I$ is a prime (resp.\ radical) ideal.
\item \label{minimal-presentation} If $\delta$ is a subdegree, then it has a unique minimal presentation as the maximum of finitely many semidegrees.
\item \label{semi-components} The non-zero semidegrees in the minimal presentation of $\delta$ are (up to integer multiples) precisely the orders of pole along the irreducible components of the hypersurface at infinity. 
\end{enumerate}
\end{thm}

\subsubsection{Factorization in terms of degree-wise Puiseux series}
Given a degree-wise Puiseux series $\psi$ in $x$, the {\em polydromy order} of $\psi$ is the smallest positive integer $p$ such that the exponents of all terms in $\psi$ are of the form $q/p$, $q \in \zz$. Let $\psi = \sum_{q \leq q_0} a_q x^{q/p}$, where $p$ is the polydromy order of $\psi$. Then the {\em conjugates} of $\psi$ are $\psi_j := \sum_{q \leq q_0} a_q \zeta^q x^{q/p}$, $1 \leq j \leq p$, where $\zeta$ is a primitive $p$-th root of unity. The usual factorization of polynomials in terms of Puiseux series implies the following

\begin{thm} \label{dpuiseux-factorization}
Let $f \in \cc[x,y]$. Then there are unique (up to conjugacy) degree-wise Puiseux series $\psi_1, \ldots, \psi_k$ and a unique non-negative integer $m$ such that
$$f = x^m \prod_{i=1}^k \prod_{\parbox{1.3cm}{\scriptsize{$\psi_{ij}$ is a con\-ju\-ga\-te of $\psi_i$}}} \left(y - \psi_{ij}(x)\right)$$
\end{thm}

\subsection{Idea of the proof} \label{sec-idea}
\begin{defn}
Let $X := \cc^2$ with coordinates $(x,y)$. Let $\phi(x)$ be a degree-wise Puiseux series in $x$ and $C \subseteq X$ be an analytic curve. We say that $(x,\phi(x))$ is a {\em parametrization of a branch of $C$ at infinity} iff there is a branch of $C$ with a parametrization of the form $t \mapsto (t, \phi(t))$ for $|t| \gg 0$.
\end{defn}

Let $\bar X$ be a normal analytic compactification of $X$ with $C_\infty := \bar X \setminus X$ irreducible and let $\delta$ be the semidegree on $\cc(x,y)$ corresponding to $C_\infty$. Let $\tilde \phi_\delta(x,\xi)$ be the generic degree-wise Puiseux series for $\delta$. The following is the key Proposition for the proof. 

\begin{prop} \label{key-prop}
Let $f_0, \ldots, f_k$ be the key forms associated to $\delta$.
\begin{enumerate}
\item \label{easy-assertion} If $f_0, \ldots, f_k$ are all polynomials, then $\bar X$ is isomorphic to the closure of the image of $X$ in the weighted projective variety $\pp^{k+1}(1,\delta(f_0), \ldots, \delta(f_k))$ under the mapping $(x,y) \mapsto [1:f_0:\cdots:f_k]$. 
\item \label{one-place-assertion} If $f_k$ is a polynomial then $C_k := V(f_k) \subseteq X$ is a curve with one place at infinity and its unique branch at infinity has a parametrization of the form \eqref{dpuiseux-param} (from Proposition \ref{prop-param}).
\item \label{hard-assertion} If there exists $j$, $0 \leq j \leq k$, such that $f_j$ is {\em not} a polynomial, then there does not exist any polynomial $f \in \cc[x,y]$ such that {\em every} branch of $V(f) \subseteq X$ at infinity has a parametrization of the form \eqref{dpuiseux-param}.
\end{enumerate}
\end{prop}

Below we use Proposition \ref{key-prop} to prove Theorems \ref{local-algebraic-answer}, \ref{semigroup-prop}, \ref{geometric-global} and \ref{algebraic-global}. In the next subsection we prove Proposition \ref{key-prop} under the additional assumption that $\tilde \phi_\delta(x,\xi)$ has at most one Puiseux pair. 

\begin{rem}
Assertions \ref{easy-assertion} and \ref{one-place-assertion} of Proposition \ref{key-prop} are more or less straightforward to see. The hard part in our proof of assertion \ref{hard-assertion} is to keep track of all the `cancellations'. However, if $\tilde \phi_\delta(x,\xi)$ has at most one Puiseux pair, then the problem is much simpler and the proof is much shorter.
\end{rem}

\begin{proof}[Proof of Theorem \ref{algebraic-global}]
Note that assertions \ref{one-place-assertion} and \ref{hard-assertion} of Proposition \ref{key-prop} imply that the {\em last} key form of $\delta$ is a polynomial iff {\em all} the key forms of $\delta$ are polynomials. Moreover, assertion \ref{easy-assertion} shows that the latter (and hence both) of the equivalent properties of the preceding sentence imply that $\bar X$ is algebraic. Therefore it only remains to show that if $\bar X$ is algebraic then all the key forms of $\delta$ are polynomials. So assume that $\bar X$ is algebraic. Let $\bar X^0 \cong \pp^2$ be the compactification of $X$ induced by the map $(x,y) \mapsto [1:x:y]$, $\sigma: \bar X \dashrightarrow \bar X^0$ be the natural bimeromorphic map, and $S$ (resp.\ $S'$) be the finite set of points of indeterminacy of $\sigma$ (resp.\ $\sigma^{-1})$. We have two cases to consider:

\paragraph{Case 1: $\sigma(C_\infty \setminus S)$ is dense in $L_\infty := \bar X^0 \setminus X$.} In this case it follows from basic geometry of bimeromorphic maps that $\sigma$ must be an isomorphism. In particular, this implies that $\delta$ is precisely the usual degree in $(x,y)$-coordinates, i.e.\ $\tilde \phi_\delta(x,\xi) = \xi x$. The theorem then follows from Example \ref{deg-dpuiseux-key}.

\paragraph{Case 2: $\sigma(C_\infty \setminus S)$ is a point $O \in L_\infty$.} In this case we are in the situation of Proposition \ref{prop-param}. In particular, $\sigma^{-1}(L_\infty\setminus S')$ is a point $P_\infty \in C_\infty$. Since $\bar X$ is algebraic, it follows that there is an algebraic curve $C \subseteq X$ such that the closure of $C$ in $\bar X$ does {\em not} intersect $P_\infty$. Proposition \ref{prop-param} then implies that {\em every} branch of $C$ at infinity has a parametrization of the form \eqref{dpuiseux-param}. Then assertion \ref{hard-assertion} of Proposition \ref{key-prop} implies that all the key forms of $\delta$ are polynomials, as required. 
\end{proof}

\begin{proof}[Proof of Theorem \ref{geometric-global}]
We continue to use the notation of the proof of Theorem \ref{algebraic-global}. Note that $\sigma'$ of Theorem \ref{geometric-global} is precisely $\sigma^{-1}$. At first assume $\bar X$ is algebraic. Since the last key form $f_k$ is a polynomial (which follows from Theorem \ref{algebraic-global}), assertion \ref{one-place-assertion} of Proposition \ref{key-prop} and Proposition \ref{prop-param} imply that $C := V(f_k) \subseteq X$ satisfies the requirement of Theorem \ref{geometric-global} and completes the proof of $(\Rightarrow)$ direction of Theorem \ref{geometric-global}. \\

Now we assume that $\bar X$ is not algebraic. Then Theorem \ref{algebraic-global} implies that one of the key polynomials is {\em not} a polynomial. It then follows from assertion \ref{hard-assertion} of Proposition \ref{key-prop} and Proposition \ref{prop-param} that $P_\infty$ lies on the closure in $\bar X$ of all algebraic curves in $X$, which completes the proof of $(\Leftarrow)$ direction of Theorem \ref{geometric-global}, as required.
\end{proof}

\begin{proof} [Proof of Theorem \ref{local-algebraic-answer}]
Recall that $L = \{u = 0\}$. Let the Puiseux expansion for $C$ at $O := (0,0)$ be
$$v = a_0u^{q/p} + a_1 u^{(q+1)/p} + \cdots$$
Let $\tilde f$ be as in Theorem \ref{local-algebraic-answer}. Then it is straightforward to see that 
\begin{align*} 
\tilde f = \begin{cases}
			0 & \text{if}\ r = 0, \\
			\text{a monic polynomial in $v$ of degree $p$} & \text{otherwise.}
		   \end{cases}
\end{align*}
Let $\nu$ be the divisorial valuation on $\cc(u,v)$ corresponding to $E^*_{L,C,r}$ (i.e.\ the {\em last} exceptional divisor in the set up of Question \ref{local-2-question}). Then the generic Puiseux series (Definition \ref{generic-puiseux-nition}) corresponding to $\nu$ is 
\begin{align} \label{tilde-phi-nu-alg}
\tilde \phi_\nu(u,\xi) = \begin{cases}
							\xi u^{q/p} & \text{if}\ r = 0, \\
							a_0u^{q/p} + \cdots + a_{r-1}u^{(q+r-1)/p} + \xi u^{(q+r)/p} & \text{otherwise.}
						   \end{cases}
\end{align}
(this is a special case of Example \ref{general-example}). If $r = 0$, then the key polynomials for $\nu$ are $U_0 = u$ and $U_1 = v$. For $r \geq 1$, the sequence continues with $U_2 =  v^p - a_0^pu^q$ and so on, with 
$$U_j = U_{j-1} - \text{a monomial term in $u,v$}\quad \text{for $j \geq 3$.}$$
It then follows from the construction of $\tilde f$ and the defining properties (and uniqueness) of key polynomials that $\tilde f$ is precisely the {\em last key polynomial} $U_k$ of $\nu$.\\

Now identify $X := \pp^2 \setminus L$ with $\cc^2$ with coordinates $(x,y) := (1/u,v/u)$. Then $\tilde E_{L,C,r}$ is algebraically contractible iff the compactification $\bar X$ of $X$ corresponding to the semidegree $\delta :=-\nu$ is algebraic. There are two cases to consider:

\paragraph{Case 1: $\tilde f =0$.} This corresponds to the case that $r = 0$. Then \eqref{tilde-phi-nu-alg} implies that $\delta$ is precisely the weighted degree corresponding to weights $p$ for $x$ and $p-q$ for $y$. It follows that $\bar X$ is the weighted projective space $\pp^2(1,p,p-q)$ and therefore $\tilde E_{L,C,r}$ is algebraically contractible, as required.

\paragraph{Case 2: $\tilde f \neq 0$.} This means $r \geq 1$ and $\tilde f$ is the last key polynomial $U_k$ of $\nu$. It is straightforward to see (e.g.\ using the uniqueness of key forms) that the last key form of $\delta$ is precisely $x^pU_k(y/x,1/x)$, and the latter is a polynomial iff $\deg_{(u,v)}(U_k) \leq p$. Theorem \ref{local-algebraic-answer} now follows from Theorem \ref{algebraic-global}.
\end{proof}

\begin{proof}[Proof of Theorem \ref{semigroup-prop}]
We use the notations of Theorem \ref{semigroup-prop} and Question \ref{local-2-question}. Set 
\begin{align*}
\tilde q_j := \begin{cases}
				p_1 \cdots p_j - q_j &\text{for}\ 1 \leq j \leq \tilde l,\\
				p_1 \cdots p_l - q_l - r &\text{for}\ j = \tilde l + 1.
			  \end{cases}
\end{align*}
Consider a generic degree-wise Puisuex series of the form
\begin{align*}
\tilde \phi_{\vec{a}} := \left(a_1 x^{\frac{\tilde q_1}{p_1}} + \sum_{j=1}^{k_1} a_{1j}x^{\frac{\tilde q_{1j}}{p_1}} \right)  +  \left(a_2 x^{\frac{\tilde q_2}{p_1p_2}} + \sum_{j=1}^{k_2} a_{2j}x^{\frac{\tilde q_{2j}}{p_1p_2}} \right)  + \cdots + \left(a_{\tilde l} x^{\frac{\tilde q_{\tilde l}}{p_1p_2 \cdots p_{\tilde l}}} + \sum_{j=1}^{k_{\tilde l}} a_{\tilde l,j}x^{\frac{\tilde q_{\tilde l,j}}{p_1p_2 \cdots p_{\tilde l}}} \right)  +  \xi x^{\frac{q_{\tilde l + 1}}{p_1p_2 \cdots p_l}}
\end{align*}
where $a_1, \ldots, a_{\tilde l} \in \cc^*$ and $a_{ij}$'s belong to $\cc$. As in the proof of Theorem \ref{local-algebraic-answer}, identify $X := \pp^2 \setminus L$ with $\cc^2$ with coordinates $(x,y) := (1/u,v/u)$. Recall that we assume in Theorem \ref{semigroup-prop} that $\tilde E_{L,C,r}$ is contractible for every curve $C$ with Puiseux pairs $(q_1, p_1), \ldots, (q_l,p_l)$. This is equivalent to saying that for all choices of $a_i$'s and $a_{ij}$'s, the semidegree $\delta_{\vec{a}}$ corresponding to $\tilde \phi_{\vec{a}}$ is the pole along the curve at infinity on some normal analytic compactification $\bar X_{\vec{a}}$ of $X$ with one irreducible curve at infinity. The statements of Theorem \ref{semigroup-prop} then translate into the following statements:
\begin{enumerate}
\item \label{algebraic-existence-2} There exist $a_i$'s and $a_{ij}$'s such that $\bar X_{\vec{a}}$ is algebraic, iff the semigroup condition \eqref{semigroup-criterion-1} holds for all $k$, $1 \leq k \leq \tilde l$.
\item \label{non-algebraic-existence-2} There exist $a_i$'s and $a_{ij}$'s such that $\bar X_{\vec{a}}$ is {\em not} algebraic iff either \eqref{semigroup-criterion-1} or \eqref{semigroup-criterion-2} {\em fails} for some $k$, $1 \leq k \leq \tilde l$.
\end{enumerate} 

At first we prove $(\Leftarrow)$ implication of Statement \ref{algebraic-existence-2}. So assume that the semigroup condition \eqref{semigroup-criterion-1} holds for all $k$, $1 \leq k \leq \tilde l$. Let $\vec{a_0}$ corresponds to the choice $a_1 = \cdots = a_{\tilde l} = 1$ and $a_{ij} = 0$ for all $i,j$. It suffices to show that $\bar X_{\vec{a_0}}$ is algebraic. Indeed, it follows from semigroup conditions \eqref{semigroup-criterion-1} that for all $k$, $1 \leq k \leq \tilde l$,
\begin{align}
p_k\tilde m_k  = \sum_{j=0}^{k-1} \beta_{k,j} \tilde m_j
\end{align}
for non-negative integers $\beta_{k,0}, \ldots, \beta_{k,k-1}$. It is then straightforward to compute that the key forms of $\delta_{\vec{a}}$ are $f_0, \ldots, f_{\tilde l+1}$, with $f_0 := x$, $f_1 := y$, and 
\begin{align}
f_{k+1} = f_k^{p_k} - c_k \prod_{j=0}^{k-1}f_j^{\beta_{k,j}},\quad c_k \in \cc^*,\ 1 \leq k \leq \tilde l.
\end{align}
In particular, each key form is a polynomial, and therefore Theorem \ref{algebraic-global} implies that $\bar X_{\vec{a_0}}$ is algebraic, as required.\\

Statement \ref{non-algebraic-existence-2} and the $(\Rightarrow)$ implication of Statement \ref{algebraic-existence-2} follow from the properties of key forms of $\delta_{\vec{a}}$ listed in the following Claim. The Claim follows from an induction on $\tilde l$ via a straightforward (but a bit messy) computation and we omit the proof.

\begin{claim*}
Let $\delta_{\vec{a}}$ be the semidegree defined as above and $f_0, \ldots, f_s$ be the key polynomials of $\delta_{\vec{a}}$. Pick the subsequence $f_{j_1}, f_{j_2}, \ldots$ of $f_j$'s consisting of all $f_{j_k}$ such that $n_{j_k} > 1$ (where $n_{j_k}$ is as in Property \ref{semigroup-property} of Definition \ref{key-defn}). Then
\begin{enumerate}
\item There are precisely $\tilde l$ of these $f_{j_k}$'s.
\item $j_{\tilde l} < s$.
\item $\delta(f_{j_k}) = \tilde m_k/\tilde p$ and $n_{j_k} = p_k$, $1 \leq k \leq \tilde l$, where 
\begin{align*}
\tilde p:= \begin{cases}
				p_l 	& \text{if}\ \tilde l = l-1, \\
				1		& \text{if}\ \tilde l > l.
			\end{cases}
\end{align*}
\item Define $j_0 := 0$, i.e.\ $f_{j_0} = x$ and $\delta(f_{j_0}) = \tilde m_0$. Then for each $k$, $1 \leq k \leq \tilde l$, 
\begin{align*}
f_{j_k+1} = f_{j_k}^{p_k} - c_k \prod_{i=0}^{k-1}f_{j_i}^{\beta_{k,i}}\quad \text{for some}\ c_k \in \cc^*.
\end{align*}
\item Define $j_{\tilde l + 1} := s$. Then $\delta(f_{j_{\tilde l + 1}}) = \delta(f_s) = \tilde m_{\tilde l+1}$. 
\item Fix $k$, $0 \leq k \leq \tilde l$. For every $i$ such that $j_k < i < j_{k+1}$, 
$$\delta(f_i) \in M_k := (\tilde m_{k+1}, p_k \tilde m_k) \cap \zz\langle \tilde m_0, \ldots, \tilde m_k \rangle$$
and on conversely, for every $m \in M_k$, there is a choice of $\vec{a}$ such that there is $i$ with $j_k < i < j_{k+1}$ and $\delta(f_i) = m$. \qed
\end{enumerate}
\end{claim*}
\noqed
\end{proof}

\subsection{Proof of Proposition \ref{key-prop} in the case of at most one Puiseux pair}

In this subsection we prove Proposition \ref{key-prop} under the assumption that the generic degree-wise Puiseux series $\tilde \phi_\delta(x,\xi)$ for $\delta$ has at most one Puiseux pair, i.e.\ it has one of the following two forms:
\begin{enumerate}
\item[Case 1:] $\tilde \phi_\delta(x, \xi) = h(x) + \xi x^r$ for some $h(x) \in \cc[x,x^{-1}]$ and $r \in \qq$, or
\item[Case 2:] $\tilde \phi_\delta(x, \xi) = h(x) + a_0x^{q/p} + a_1 x^{(q-1)/p} + \cdots + a_{s-1} x^{(q-s+1)/p} + \xi x^{(q-s)/p}$ for some $h(x) \in \cc[x,x^{-1}]$ and $p,q,s \in \zz$ such that $p \geq 2$, $s \geq 1$, $q \neq 0$ and $p,q$ are co-prime.
\end{enumerate} 
Assume we are in Case 1. We claim that $h(x) \in \cc[x]$. Indeed, otherwise we have $h(x) = h_0(x) + x^{-1}h_1(x)$ for some $h_0 \in \cc[x]$ and $h_1 \neq 0 \in \cc[x^{-1}]$, and it would follow from \eqref{phi-delta-defn} that $\delta(y-h_0(x)) < 0$, which is impossible, since $\delta$ takes positive value on all non-constant polynomials. Similarly, we must have $r > 0$. Let $h(x) = \sum_{i=1}^m b_i x^{d_i}$ with $d_1 > \cdots > d_m > r$. It follows that the key forms of $\delta$ are precisely, $x,y, y - b_1 x^{d_1}, y - b_1 x^{d_1} - b_2 x^{d_2}, \ldots, y - b_1 x^{d_1} - \cdots - b_m x^{d_m}$. Since the `last' key form is $y - h(x)$ and all key forms are polynomials, assertions \ref{one-place-assertion} and \ref{hard-assertion} of Proposition \ref{key-prop} are automatically satisfied. For assertion \ref{easy-assertion}, set $y' := y - h(x)$ and observe that $\delta$ is a weighted degree in $(x,y')$-coordinates corresponding to weights $p$ for $x$ and $q$ for $y'$, where $r = q/p$ with $p,q$ co-prime. It follows that $\bar X$ is precisely the weighted projective space $\pp^2(1,p,q)$ with the embedding $X \into \bar X$ given by $(x,y') \mapsto [1:x:y']$. Assertion \ref{easy-assertion} of Proposition \ref{key-prop} now follows in a straightforward way.\\

Now assume Case 2 holds. It follows as in Case 1 that $h(x) \in \cc[x]$ and $q > 0$. Moreover, letting $h(x) = \sum_{i=1}^m a_i x^{d_i}$ with $d_1 > \cdots > d_m > q/p$, we have that the first $m+2$ key forms of $\delta$ are $x,y, y - b_1 x^{d_1}, \ldots, y - h(x)$. Since these are already polynomials, it follows from the definition of key forms that in order to prove Proposition \ref{key-prop}, w.l.o.g.\ we may apply the change of coordinates $(x,y) \mapsto (x,y-h(x))$ and assume that 
$$\tilde \phi_\delta(x, \xi) = a_0x^{q/p} + a_1 x^{(q-1)/p} + \cdots + a_{s-1} x^{(q-s+1)/p} + \xi x^{(q-s)/p}.$$  
We now compute the key forms of $\omega$. Define
$$\Phi_\delta(x,y) := \prod_{j=1}^p \left(y - a_0\zeta^{jq} x^{q/p} - a_1 \zeta^{j(q-1)}x^{(q-1)/p} - \cdots - a_{s-1} \zeta^{j(q-s+1)}x^{(q-s+1)/p}\right),$$  
where $\zeta$ is a primitive $p$-th root of unity. In other words, $\Phi_\delta$ is the unique monic polynomial in $y$ with coefficients in $\cc[x,x^{-1}]$ whose roots are {\em conjugates} of $a_0x^{q/p} + a_1 x^{(q-1)/p} + \cdots + a_{s-1} x^{(q-s+1)/p}$. Let $d_\delta := \delta(\Phi_\delta)$. Then
\begin{align*}
d_\delta = \delta(x) \ord_x\left(\Phi_\delta(x,y)|_{y = \tilde \phi_\delta(x,\xi)}\right) = (p-1)q + q-s = pq-s.
\end{align*}
Let $\omega$ be the weighted degree on $\cc(x,y)$ which gives weight $p$ to $x$ and $q$ to $y$. Note that 
$$\Phi_\delta = y^p - a_0^p x^q - \sum_j g_j$$
where each $g_j$'s are monomial terms of the form $c_jx^{\alpha_j}y^{\beta_j}$ for some $c_j \in \cc$ and integers $\alpha_j, \beta_j$ such that $0 \leq \beta_j < p$. Order the $g_j$'s so that $\omega(g_1) \geq \omega(g_2) \geq \cdots \geq \omega(g_m) > d_\delta \geq \omega(g_{m+1}) \geq \cdots$.

\begin{claim} \label{claim}
The key forms of $\delta$ are $x,y, y^p - a_0^p x^q, y^p - a_0^p x^q - g_1, \ldots, y^p - a_0^p x^q - \sum_{j=1}^m g_j$. 
\end{claim}

\begin{proof}
Let $h_0 := y^p - a_0^p x^q$ and $h_j := y^p - a_0^p x^q - \sum_{i=1}^jg_i $ for $1 \leq j \leq m$. It follows from the definition of $\Phi_\delta$ that 
\begin{align}
h_m|_{y = \tilde \phi_\delta(x,\xi)} = (c\xi -c') x^{(pq-s)/p} + \lot \label{last-key}
\end{align}
for some $c, c' \in \cc$, $c \neq 0$, which implies that $\delta(h_m) = d_\delta$. A straightforward backward induction then proves that $\delta(h_m) < \delta(h_{m-1}) < \cdots < \delta(h_0)$. The uniqueness of key forms then imply that $x,y, h_0, \ldots, h_m$ are key forms for $\delta$. Moreover, note that the leading term of the right hand side of identity \eqref{last-key} contains the indeterminate $\xi$, which implies that for any $n \geq 1$, the value of $\delta((h_m)^n)$ can {\em not} be reduced via adding any polynomial in $x, x^{-1}, y, h_0, \ldots, h_{m-1}$. In particular, $h_m$ is the `last' key form for $\delta$.
\end{proof}

\begin{proof}[Proof of assertion \ref{easy-assertion} of Proposition \ref{key-prop}]
Assume that $h_j \in \cc[x,y]$ for each $j$, $0 \leq j \leq k$. Let $\omega_j := \omega(h_j)$ for $0 \leq j \leq m$. Let $\WP$ be the weighted projective space $\pp(1,p,q,\omega_0, \ldots, \omega_m)$ with weighted homogeneous coordinates $[z:x:y:y_0: \cdots: y_m]$. We have to show that $\bar X$ is isomorphic to the closure in $\WP$ of the image of $X$ under the embedding  $(x,y) \mapsto [1:x:y:h_0: \cdots: h_m]$. Let $R$ be the {\em homogeneous coordinate ring} of $\WP$, i.e.\ $R := \cc[z,x,y,y_0, \ldots, y_m]$ with the grading on $R$ given by the weights $1$ for $z$, $p$ for $x$, $q$ for $y$, and $\omega_j$ for $y_j$, $0 \leq j \leq m$. Note that the closure $Z$ of $X$ in $\WP$ is naturally isomorphic to $\proj R/J$ for some homogeneous ideal $J$ of $R$. Consequently, we have to show that $\proj R/J \cong \proj \cc[x,y]^\delta$, where the latter ring is the graded ring corresponding to $\delta$ as in Subsubsection \ref{degree-like-section}. For this it suffices to show that the graded $\cc$-algebra homomorphism $\cc[z,x,y,y_0, \ldots, y_m] \to \cc[x,y]^{\delta}$ which maps $z \mapsto (1)_1$, $x \mapsto (x)_p$, $y \mapsto (y)_q$ and $y_j \mapsto (h_j)_{\omega_j}$ is in fact a surjection. But the latter is an immediate consequence of Property \eqref{generating-property} of key forms. This completes the proof of assertion \ref{easy-assertion} of Proposition \ref{key-prop}. 
\end{proof}

\begin{proof}[Proof of assertion \ref{one-place-assertion} of Proposition \ref{key-prop}]
Note that $f_k$ of Proposition \ref{key-prop} is precisely $h_m$. Observe that
\begin{compactenum}
\item $h_m$ is a monic polynomial in $y$ of degree $p$.
\item Since $h_m$ is a polynomial by assumption, it is also a monic polynomial in $x$ of degree $q$.
\item $\omega(h_m) = pq$.
\end{compactenum}
Since $p$ and $q$ are relatively prime, these observations imply that $h_m$ has one place at infinity and there is a degree-wise Puiseux series $\psi(x)$ with polydromy order $p$ such that 
\begin{align}
h_m = \prod_{j=1}^p \left(y - \psi_j(x)\right), \label{h-m-factorization}
\end{align}
$\psi_j$'s are the conjugates of $\psi$. Identities \eqref{h-m-factorization} and \eqref{last-key} then imply that there must be $j$, $1 \leq j \leq p$, such that 
$$\psi_j(x) = \tilde \phi_\delta(x,\xi)|_{\xi = c''} + \lot$$
for some $c'' \in \cc$. This proves that the curve of $h_m$ has a parametrization at infinity of the form \eqref{dpuiseux-param}, as required.
\end{proof}

\begin{proof}[Proof of assertion \ref{hard-assertion} of Proposition \ref{key-prop}]
We prove it by contradiction. So assume there exists $j$, $1 \leq j \leq m$, such that $\alpha_j < 0$ and that there exists a polynomial $f \in \cc[x,y]$ such that all of its branches at infinity has a parametrization of the form 
$$t \mapsto (t, a_0t^{q/p} + a_1 t^{(q-1)/p} + \cdots + a_{s-1} t^{(q-s+1)/p} + c t^{(q-s)/p} + \lot)$$
for some $c \in \cc$ (where $c$ depends on the branch). Let us write $\phi(x) := a_0x^{q/p} + a_1 x^{(q-1)/p} + \cdots + a_{s-1} x^{(q-s+1)/p}$. Then it follows from Theorem \ref{dpuiseux-factorization} that there exist degree-wise Puiseux series $\psi_1, \ldots, \psi_l$ in $x$ such that $f$ has a factorization of the form
\begin{align*}
f = \prod_{i=1}^l \prod_{\parbox{1.3cm}{\scriptsize{$\psi_{ij}$ is a con\-ju\-ga\-te of $\psi_i$}}} \left(y - \psi_{ij}(x)\right)
\end{align*}
where for each $i$, $1 \leq i \leq l$,
\begin{align}
\psi_i(x) = \phi(x) + c_i x^{(q-s)/p} + \lot \label{initial-i}
\end{align}
for some $c_i \in \cc$. Pick $i$, $1 \leq i \leq l$, and let 
$$\Psi_i := \prod_j \left(y - \psi_{ij}(x)\right).$$
It follows from \eqref{initial-i} that $p$ divides the polydromy order $p_i$ (which is also the number of conjugates) of $\psi_i$, and for each conjugate $\psi_{ij}$ of $\psi_i$, there is a conjugate $\phi_{j'}$ of $\phi$ such that 
$$\psi_{ij}(x) = \phi_{j'}(x) + c_{ij} x^{(q-s)/p} + \lot$$
for some $c_{ij} \in \cc$. It follows that $\Psi_i$ can be expressed in the following form:
\begin{align*}
\Psi_i 		&= \prod_{j'=1}^p \prod_{j=1}^{p_i/p} \left(y - \phi_{j'}(x) - \tilde \psi_{j'j}(x)\right),\quad \text{where}\ \deg_x\left(\tilde \psi_{j'j}(x)\right) \leq (q-s)/p, \\
			&= \prod_{j=1}^{p_i/p} \Psi_{ij},\quad \text{where} \\
\Psi_{ij}	&:= \prod_{j'=1}^p \left(y - \phi_{j'}(x) - \tilde \psi_{j'j}(x)\right),\ \text{for each} \ j,\ 1 \leq j \leq p_i/p. 
\end{align*}
Fix a $j$, $1 \leq j \leq p_i/p$. Let $\omega$ be the weighted degree defined preceding Claim \ref{claim}. Note that $\omega$ extends to a weighted degree on $\dpsxc[y]$ (corresponding to weight $p$ for $x$ and $q$ for $y$), where $\dpsxc$ is the field of degree-wise Puiseux series in $x$. Then it follows that 
\begin{align*}
\Psi_{ij} &= \prod_{j'=1}^p \left(y - \phi_{j'}(x)\right) + H_{ij}(x,y)\quad \text{for some}\ H_{ij} \in \dpsxc[y],\ \omega(H_{ij}) \leq  pq-s,\\
	&= \Phi_\delta(x,y) + H_{ij}(x,y) \\
	&= y^p - a_0^p x^q - \sum_{k=1}^m g_k + \tilde H_{ij}(x,y) \quad \text{for some}\ \tilde H_{ij} \in \dpsxc[y],\ \omega(\tilde H_{ij}) \leq  pq-s.
\end{align*}
Now we prepare for the contradiction. Pick the smallest integer $k_0$ such that $\alpha_{k_0} < 0$ and let $W_0 := \omega(g_{k_0}) > pq -s$. Collecting all terms of $\Psi_{ij}$ with $\omega$ value less than $W_0$ yields:
\begin{align*}
\Psi_{ij} = y^p - a_0^p x^q - \sum_{k=1}^{k_0} g_k + G_{ij}(x,y) \quad \text{for some}\ G_{ij} \in \dpsxc[y],\ \omega(G_{ij}) < W_0.
\end{align*}
In particular, note that $\Psi_{ij} - G_{ij}$ is independent of $i,j$. Now
\begin{align*}
f &= \prod_{i,j} \Psi_{ij} = \prod_{i,j} \left(y^p - a_0^p x^q - \sum_{k=1}^{k_0} g_k + G_{ij}(x,y) \right).
\end{align*}
Let $M$ be the total number of factors in the product of right hand side, and for each $W \in \qq$, let $f_W$ be the sum of monomials that appear (after multiplying out all the factors) in the right hand side with $\omega$-value equal to $W$. Then for all $W > W_1 := (M-1)pq + W_0$, $f_W$ is a polynomial in $x$ and $y$. Moreover, $f_{W_1} = \tilde f - M(y^p -  a_0^{p}x^{q})^{M-1}g_{k_0}$ for a polynomial $\tilde f$ in $x$ and $y$. Let $f' := f - \sum_{W > W_1} f_W - \tilde f$. Then it follows that $f'$ is a polynomial with $\omega(f) = W_1$, but the leading weighted homogeneous form (with respect to $\omega$) of $f'$ is $- M(y^p - a_0^{p}x^{q})^{M-1}g_{k_0}$, which is {\em not} a polynomial. This gives the desired contradiction and completes the proof of assertion \ref{hard-assertion} of Proposition \ref{key-prop}.
\end{proof}

\bibliographystyle{alpha}
\bibliography{bibi}

\end{document}

%% file: commands-2012.tex
\makeatletter

\newcommand{\Rmnum}[1]{\expandafter\@slowromancap\romannumeral #1@}
\makeatother

\DeclareMathOperator\gr{gr}
 
\DeclareMathOperator\ord{ord}

\DeclareMathOperator\proj{Proj}

\newcommand{\scrF}{\ensuremath{\mathcal{F}}}

\newcommand{\scrV}{\ensuremath{\mathcal{V}}}

\newcommand{\cc}{\ensuremath{\mathbb{C}}}
\newcommand{\kk}{\ensuremath{\mathbb{K}}}

\newcommand{\pp}{\ensuremath{\mathbb{P}}}
\newcommand{\qq}{\ensuremath{\mathbb{Q}}}
\newcommand{\rr}{\ensuremath{\mathbb{R}}}
\newcommand{\zz}{\ensuremath{\mathbb{Z}}}

\newcommand{\sheaf}{\ensuremath{\mathcal{O}}}

\newcommand{\dsum}{\ensuremath{\bigoplus}}
\newcommand{\into}{\ensuremath{\hookrightarrow}}

\newtheorem{thm}{Theorem}[section]
\newtheorem*{thm*}{Theorem}
\newtheorem{lemma}[thm]{Lemma}
\newtheorem*{lemma*}{Lemma}
\newtheorem{lemma-in-thm}{Lemma}[thm]

\newtheorem{prop}[thm]{Proposition}
\newtheorem*{prop*}{Proposition}
\newtheorem{cor}[thm]{Corollary}
\newtheorem{claim}[thm]{Claim}
\newtheorem*{claim*}{Claim}

\newtheorem*{conjecture*}{Conjecture}

\theoremstyle{definition} 

\newtheorem{defn}[thm]{Definition}
\newtheorem*{defn*}{Definition}

\newtheorem*{definotation*}{Definition-Notation}
\newtheorem{example}[thm]{Example}
\newtheorem*{example*}{Example}

\newtheorem*{fact*}{Fact}

\newtheorem*{facts*}{Facts}

\newtheorem*{bold-note*}{Note}
\newtheorem{bold-question}[thm]{Question}
\newtheorem{rem}[thm]{Remark}

\newtheorem*{reminition*}{Remark-Definition}
\newtheorem{remexample}[thm]{Remark-Example}
\newtheorem{remexample*}{Remark-Example}

\newtheorem*{remtation*}{Remark-Notation}

\newtheorem*{remuestion*}{Remark-Question}

\newtheorem*{constrinition*}{Construction-Definition}

\theoremstyle{remark}
\newtheorem*{rem*}{Remark}
\newtheorem*{note*}{Note}
\newtheorem*{notation*}{Notation}
\newtheorem*{question*}{Question}
\newtheorem{question}[thm]{Question}
\newtheorem*{questions*}{Questions}

\newcounter{UnorderedProofTempCtr}
\newcommand{\tempcommand}{}

\newcommand{\WP}{\ensuremath{{\bf W}\pp}}